\documentclass[11 pt, twoside, reqno]{amsart}
\usepackage{parskip}
\usepackage{amsthm,amsmath,amssymb,oldgerm}
\usepackage[a4paper, margin=2.5cm]{geometry}
\usepackage{mathrsfs}
\usepackage{verbatim}
\usepackage{hyperref}
\usepackage{cleveref}
\theoremstyle{plain}
\newtheorem{thm}{Theorem}[section]
\newtheorem*{theorem*}{Main theorem}
\newtheorem{mainthm}[]{Theorem}[]

\newtheorem{lem}[thm]{Lemma}
\newtheorem{prop}[thm]{Proposition}

\theoremstyle{definition}
\theoremstyle{definition}
\newtheorem{rem}[thm]{Remark}
\newcommand{\dr}{\,\mathrm{d}r_}
\newcommand{\dd}{\,\mathrm{d}}

\newcommand{\RR}{\mathbb{R}}

\title[Off-diagonal estimates of the Bergman kernel associated to Siegel varieties]{Off-diagonal estimates of Bergman kernel associated to Siegel varieties}
\author{Anilatmaja Aryasomayajula}
\address{Department of Mathematics, Indian Institute of Science Education and Research (IISER) Tirupati, 
Transit campus at Sri Rama Engineering College, Karkambadi Road,
Mangalam (B.O),Tirupati-517507, India.}
\email{anil.arya@iisertirupati.ac.in}
\author{Harinarayanan G}
\address{Department of Mathematics, Indian Institute of Science Education and Research (IISER) Tirupati, 
Transit campus at Sri Rama Engineering College, Karkambadi Road,
Mangalam (B.O),Tirupati-517507, India.}
\email{harinarayanan@students.iisertirupati.ac.in}
\date{\today}
\subjclass[2020]{}
\begin{document}
%%%%%%%%%%%%%%%%%%%%%%%%%%%%%%%%%%%%%%%%%%%%%%%%%%%%%%%%%%%%%%%%%%%%%%%%%%%%%%%%%%%%%%%%%%
\begin{abstract}
%%%%%%%%%%%%%%%%%%%%%%%%%%%%%%%%%%%%%%%%%%%%%%%%%%%%%%%%%%%%%%%%%%%%%%%%%%%%%%%%%%%%%%%%%%
For $g\geq 2$, let $\Gamma\subset\mathrm{Sp}(2g,\mathbb{R})$ be a discrete subgroup, which is either a cocompact subgroup or an arithmetic subgroup without torsion elements, and let $\mathbb{H}_{g}$ denote the Siegel upper half space of genus $g$. Let $X_{\Gamma}:=\Gamma\backslash\mathbb{H}_{g}$ denote the quotient space, which is a complex manifold of dimension $g(g+1)/2$. Let $\Omega_{X_{\Gamma}}$ denote the cotangent bundle, and let $\ell:=\mathrm{det}(\Omega_{X_{\Gamma}})$ denote the determinant line bundle of $\Omega_{X_{\Gamma}}$. For any  $Z,W\in X_{\Gamma}$, let $d_{\mathrm{S}}(Z,W)$ denote the geodesic distance between the points $Z$ and $W$ on $X_{\Gamma}$.
%%%%%%%%%%%%%%%%%%%%%%%%%%%%%%%%%%%%%%%%%%%%%%%%%%%%%%%%%%%%%%%%%%%%%%%%%%%%%%%%%%%%%%%%%%

\vspace{0.15cm}\noindent
For any $k\geq 1$, let $H^{0}(X_{\Gamma},\ell^{\otimes k})$ denote the complex vector space of global sections of the line bundle $\ell^{\otimes k}$, and let $\|\cdot\|_{k}$ denote the point-wise norm on $\ell^{\otimes k}$. Let $\mathcal{B}_{X_{\Gamma}}^{\ell^{ k}}$ denote the Bergman kernel associated to $H^{0}_{L^{2}}(X_{\Gamma},\ell^{\otimes k})\subset H^{0}(X_{\Gamma},\ell^{\otimes k})$, vector subspace of $L^2$ global sections. For any $k\gg 1$, and $Z=X+iY,\,W=U+iV\in X_{\Gamma}$ with $\mathrm{det}(V)>\mathrm{det}(Y)$ (we identify $X_{\Gamma}$ with its universal cover $\mathbb{H}_{g}$), when $\Gamma$ is cocompact, we have the following estimate
\begin{align*}
\|\mathcal{B}_{X_{\Gamma}}^{\ell^{k}}(Z,W)\|_{\ell^{ k}}=O_{X_{\Gamma}}\bigg(\frac{k^{g(g+1)/2}}{\cosh^{k(g+1)-g^{2}-g}(d_{\mathrm{S}}(Z,W)/2\sqrt{2})}\bigg);
\end{align*}
%%%%%%%%%%%%%%%%%%%%%%%%%%%%%%%%%%%%%%%%%%%%%%%%%%%%%%%%%%%%%%%%%%%%%%%%%%%%%%%%%%%%%%%%%%
when $\Gamma$ is an arithmetic subgroup, we have the following estimate
\begin{align*}
\|\mathcal{B}_{X_{\Gamma}}^{\ell^{ k}}(Z,W)\|_{\ell^{ k}}=O_{X_{\Gamma}}\bigg(\frac{k^{g(g+1)/4}( \mathrm{det}{(4Y)})^{k(g+1)/2}}{\big(\mathrm{det}(V) \big)^{(k(g+1)-g-1)/2}}+\frac{k^{g(g+1)/2}}{\cosh^{k(g+1)-g^{2}-g}(d_{\mathrm{S}}(Z,W)/2\sqrt{2})}\bigg),
\end{align*}
%%%%%%%%%%%%%%%%%%%%%%%%%%%%%%%%%%%%%%%%%%%%%%%%%%%%%%%%%%%%%%%%%%%%%%%%%%%%%%%%%%%%%%%%%%

\noindent
 where the implied constants in both the above estimates depend only on $X_{\Gamma}$.
%%%%%%%%%%%%%%%%%%%%%%%%%%%%%%%%%%%%%%%%%%%%%%%%%%%%%%%%%%%%%%%%%%%%%%%%%%%%%%%%%%%%%%%%%%
\end{abstract}
%%%%%%%%%%%%%%%%%%%%%%%%%%%%%%%%%%%%%%%%%%%%%%%%%%%%%%%%%%%%%%%%%%%%%%%%%%%%%%%%%%%%%%%%%%
\maketitle
%%%%%%%%%%%%%%%%%%%%%%%%%%%%%%%%%%%%%%%%%%%%%%%%%%%%%%%%%%%%%%%%%%%%%%%%%%%%%%%%%%%%%%%%%%
%%%%%%%%%%%%%%%%%%%%%%%%%%%%%%%%%%%%%%%%%%%%%%%%%%%%%%%%%%%%%%%%%%%%%%%%%%%%%%%%%%%%%%%%%%
\section{Introduction}\label{sec-1}
%%%%%%%%%%%%%%%%%%%%%%%%%%%%%%%%%%%%%%%%%%%%%%%%%%%%%%%%%%%%%%%%%%%%%%%%%%%%%%%%%%%%%%%%%%
In this section, we briefly recall the history, background, and relevance associated with our main results. We then describe our main results, and discuss our strategy.
%%%%%%%%%%%%%%%%%%%%%%%%%%%%%%%%%%%%%%%%%%%%%%%%%%%%%%%%%%%%%%%%%%%%%%%%%%%%%%%%%%%%%%%%%%

\vspace{0.2cm}
\subsection{History and background of the problem}\label{sec-1.1}
%%%%%%%%%%%%%%%%%%%%%%%%%%%%%%%%%%%%%%%%%%%%%%%%%%%%%%%%%%%%%%%%%%%%%%%%%%%%%%%%%%%%%%%%%%
Estimates of Bergman kernels associated to high tensor powers of holomorphic manifolds defined over complex manifolds is a subject of great interest in complex geometry. Optimal
estimates are derived in the setting of compact complex manifolds, by the likes of Tian , Zelditch, Catalan, et al., and more recently by Ma and Marinescu.
%%%%%%%%%%%%%%%%%%%%%%%%%%%%%%%%%%%%%%%%%%%%%%%%%%%%%%%%%%%%%%%%%%%%%%%%%%%%%%%%%%%%%%%%%%

However, off-diagonal estimates even in the setting of compact complex manifolds are difficult to derive. We refer the reader to  \cite{delin}, \cite{christ}, \cite{luzel}, and \cite{dai} for off-diagonal estimates of Bergman kernels associated to high tensor powers of holomorphic line bundles defined over a complex manifold. In  \cite{ma}, Ma and Marinescu have also derived off-diagonal estimates of the Bergman kernel, and the estimates are the sharpest in the most general context of an arbitrary compact complex manifold.
%%%%%%%%%%%%%%%%%%%%%%%%%%%%%%%%%%%%%%%%%%%%%%%%%%%%%%%%%%%%%%%%%%%%%%%%%%%%%%%%%%%%%%%%%%

In \cite{anil1} and \cite{anil2}, the first named author and Majumder have derived off diagonal estimates of the Bergman kernel associated to high tensor powers of the cotangent bundle defined over a hyperbolic Riemann surface of finite volume, both in the compact and noncompact setting. Furthermore, in \cite{anil3}, the first named author along with Roy and Sadhukhan extended the estimates to compact hyperbolic ball quotients.
%%%%%%%%%%%%%%%%%%%%%%%%%%%%%%%%%%%%%%%%%%%%%%%%%%%%%%%%%%%%%%%%%%%%%%%%%%%%%%%%%%%%%%%%%%

In this article, we extend estimates to Siegel modular varieties of finite volume, both in the compact setting, and the noncompact setting emanating from arithmetic subgroups of 
$\mathrm{Sp}(2g,\mathbb{R})$. These estimates will come handy in a future article \cite{anil4}, where we derive sub-convexity estimates of Siegel cusp forms of genus two, associated to the full modular group $\mathrm{Sp}(4,\mathbb{Z})$.
%%%%%%%%%%%%%%%%%%%%%%%%%%%%%%%%%%%%%%%%%%%%%%%%%%%%%%%%%%%%%%%%%%%%%%%%%%%%%%%%%%%%%%%%%%
%%%%%%%%%%%%%%%%%%%%%%%%%%%%%%%%%%%%%%%%%%%%%%%%%%%%%%%%%%%%%%%%%%%%%%%%%%%%%%%%%%%%%%%%%%

\vspace{0.2cm}
\subsection{Statements of results}\label{sec-1.2}
%%%%%%%%%%%%%%%%%%%%%%%%%%%%%%%%%%%%%%%%%%%%%%%%%%%%%%%%%%%%%%%%%%%%%%%%%%%%%%%%%%%%%%%%%%
We now describe the two main results of the article, which are proved in section \ref{sec-3.2}. So, we briefly introduce the notation, which is required to state the  results. We will elaborately discuss the notation and the background material necessary to prove the main results, in sections \ref{sec-2.1}--\ref{sec-2.4}.
%%%%%%%%%%%%%%%%%%%%%%%%%%%%%%%%%%%%%%%%%%%%%%%%%%%%%%%%%%%%%%%%%%%%%%%%%%%%%%%%%%%%%%%%%%

Let $\mathbb{H}_{g}$ denote the Siegel upper half space of genus $g$. Let $\Gamma\subset \mathrm{Sp(2g,\mathbb{R})}$ be a cocompact subgroup, or an arithmetic subgroup without elliptic fixed points. Let $\Omega_{X_{\Gamma}}$ denote cotangent bundle, and let $\ell:=\mathrm{det}(\Omega_{X_{\Gamma}})$ denote the determinant line bundle of $\Omega_{X_{\Gamma}}$, also known as the canonical line bundle. For any $k\geq 1$, let $H^{0}(X_{\Gamma}, \ell^{\otimes k})$ denote complex vector subspace of global sections of $\ell^{\otimes k}$. Let $\| \cdot\|_{\ell^{k}}$ and $\|\cdot\|_{\ell^{k},L^{2}}$  denote the point-wise and $L^2$ norms on $H^{0}(X_{\Gamma},\ell^{\otimes k})$, respectively, and let $H^{0}_{L^{2}}(X_{\Gamma},\ell^{\otimes k})\subset H^{0}(X_{\Gamma},\ell^{\otimes k})$ denote the subspace of $L^2$ sections. Let $\mathcal{B}_{X_{\Gamma}}^{\ell^k}$ denote the Bergman kernel of $H^{0}_{L^{2}}(X_{\Gamma},\ell^{\otimes k})$.
%%%%%%%%%%%%%%%%%%%%%%%%%%%%%%%%%%%%%%%%%%%%%%%%%%%%%%%%%%%%%%%%%%%%%%%%%%%%%%%%%%%%%%%%%%

\vspace{0.1cm}
The main results of the article are the following two theorems.
%%%%%%%%%%%%%%%%%%%%%%%%%%%%%%%%%%%%%%%%%%%%%%%%%%%%%%%%%%%%%%%%%%%%%%%%%%%%%%%%%%%%%%%%%%

\vspace{0.1cm}
\begin{mainthm}\label{mainthm1}
Let $\Gamma\subset \mathrm{Sp}(2g,\mathbb{R})$ be a cocompact subgroup. With notation as above, for $g\geq 2$ and $k\gg 1$, we have the following estimate
\begin{align}\label{mainthm1-eqn}
\|\mathcal{B}_{X_{\Gamma}}^{\ell^{k}}(Z,W)\|_{\ell^{ k}}=O_{X_{\Gamma}}\bigg(\frac{k^{g(g+1)/2}}{\cosh^{k(g+1)-g^{2}-g}(d_{\mathrm{S}}(Z,W)/2\sqrt{2})}\bigg),
\end{align}
%%%%%%%%%%%%%%%%%%%%%%%%%%%%%%%%%%%%%%%%%%%%%%%%%%%%%%%%%%%%%%%%%%%%%%%%%%%%%%%%%%%%%%%%%%
where the implied constant depends only on $X_{\Gamma}$.
\end{mainthm}
%%%%%%%%%%%%%%%%%%%%%%%%%%%%%%%%%%%%%%%%%%%%%%%%%%%%%%%%%%%%%%%%%%%%%%%%%%%%%%%%%%%%%%%%%%
%%%%%%%%%%%%%%%%%%%%%%%%%%%%%%%%%%%%%%%%%%%%%%%%%%%%%%%%%%%%%%%%%%%%%%%%%%%%%%%%%%%%%%%%%%

\vspace{0.1cm}
\begin{mainthm}\label{mainthm2}
Let $\Gamma$ be an arithmetic subgroup. With notation as above, for $g\geq 2$ and $k\gg 1$, and for $Z=X+iY,\,W=U+iV\in X_{\Gamma}$ with $\mathrm{det}(V)>\mathrm{det}(Y)$ (we identify $X_{\Gamma}$ with its universal cover $\mathbb{H}_{g}$),   we have the following estimate
\begin{align}\label{mainthm2-eqn}
\|\mathcal{B}_{X_{\Gamma}}^{\ell^{ k}}(Z,W)\|_{\ell^{ k}}=O_{X_{\Gamma}}\bigg(\frac{k^{g(g+1)/4}( \mathrm{det}{(4Y)})^{k(g+1)/2}}{\big(\mathrm{det}(V) \big)^{(k(g+1)-g-1)/2}}+\frac{k^{g(g+1)/2}}{\cosh^{k(g+1)-g^{2}-g}(d_{\mathrm{S}}(Z,W)/2\sqrt{2})}\bigg),
\end{align}
%%%%%%%%%%%%%%%%%%%%%%%%%%%%%%%%%%%%%%%%%%%%%%%%%%%%%%%%%%%%%%%%%%%%%%%%%%%%%%%%%%%%%%%%%%
where the implied constant depends only on $X_{\Gamma}$.
\end{mainthm}
%%%%%%%%%%%%%%%%%%%%%%%%%%%%%%%%%%%%%%%%%%%%%%%%%%%%%%%%%%%%%%%%%%%%%%%%%%%%%%%%%%%%%%%%%%
%%%%%%%%%%%%%%%%%%%%%%%%%%%%%%%%%%%%%%%%%%%%%%%%%%%%%%%%%%%%%%%%%%%%%%%%%%%%%%%%%%%%%%%%%%

\vspace{0.1cm}
\begin{rem}
 Theorems \ref{mainthm1} and \ref{mainthm2} are clearly an extension proved in \cite{anil1} and \cite{anil2}, to the setting of Siegel modular varieties. Furthermore, as in the setting of Riemann surfaces, our estimates are stronger than the off-diagonal estimate for the Bergman kernel proved in \cite{ma}. However, it is be noted that the estimate proved in \cite{ma} is a more general estimate, and valid for all compact complex manifolds.
 %%%%%%%%%%%%%%%%%%%%%%%%%%%%%%%%%%%%%%%%%%%%%%%%%%%%%%%%%%%%%%%%%%%%%%%%%%%%%%%%%%%%%%%%%%

\vspace{0.1cm}
Moreover, as the point wise norm of the Bergman kernel is symmetric in $Z=X+iY,\,W=U+iV\in X_{\Gamma}$ (we identify $X_{\Gamma}$ with its universal cover $\mathbb{H}_{g}$, we can also derive the following estimate from our method
\begin{align*}
\|\mathcal{B}_{X_{\Gamma}}^{\ell^{ k}}(Z,W)\|_{\ell^{ k}}=O_{X_{\Gamma}}\bigg(\frac{k^{g(g+1)/4}( \mathrm{det}{(4V)})^{k(g+1)/2}}{\big(\mathrm{det}(Y) \big)^{(k(g+1)-g-1)/2}}+\frac{k^{g(g+1)/2}}{\cosh^{k(g+1)-g^{2}-g}(d_{\mathrm{S}}(z,w)/2\sqrt{2})}\bigg),
\end{align*}
%%%%%%%%%%%%%%%%%%%%%%%%%%%%%%%%%%%%%%%%%%%%%%%%%%%%%%%%%%%%%%%%%%%%%%%%%%%%%%%%%%%%%%%%%%
where the implied constant depends only on $X_{\Gamma}$.
%%%%%%%%%%%%%%%%%%%%%%%%%%%%%%%%%%%%%%%%%%%%%%%%%%%%%%%%%%%%%%%%%%%%%%%%%%%%%%%%%%%%%%%%%%

Moreover, estimate \eqref{mainthm2-eqn} describes the asymptotic behavior of the Bergman kernel, as one of the variables $Z$ and $W$, approaches the boundary $\mathcal{D}_{\mathrm{bdy}}$ (see equation \eqref{bdy} in section \ref{sec-2.2} for description of the boundary divisor $\mathcal{D}_{\mathrm{bdy}}$), and the other remains bounded. 
%%%%%%%%%%%%%%%%%%%%%%%%%%%%%%%%%%%%%%%%%%%%%%%%%%%%%%%%%%%%%%%%%%%%%%%%%%%%%%%%%%%%%%%%%%

\vspace{0.1cm}
Furthermore, when $\Gamma$ is cocompact and $Z=W$, we immediately recover the well established diagonal estimate of the Bergman kernel, i. e.,
\begin{align*}
\|\mathcal{B}_{X_{\Gamma}}^{\ell^{ k}}(Z,Z)\|_{\ell^{ k}}=O_{X_{\Gamma}}\big(k^{g(g+1)/2}\big),
\end{align*}
%%%%%%%%%%%%%%%%%%%%%%%%%%%%%%%%%%%%%%%%%%%%%%%%%%%%%%%%%%%%%%%%%%%%%%%%%%%%%%%%%%%%%%%%%%
where the implied constant depends only on $X_{\Gamma}$.
%%%%%%%%%%%%%%%%%%%%%%%%%%%%%%%%%%%%%%%%%%%%%%%%%%%%%%%%%%%%%%%%%%%%%%%%%%%%%%%%%%%%%%%%%%

\vspace{0.1cm}
When $\Gamma$ is an arithmetic subgroup and $Z=W$, adapting arguments from \cite{mandal} about the domain where the Bergman kernel attains its maximum, we recover the diagonal estimate of Kramer and mandal
\begin{align*}
\|\mathcal{B}_{X_{\Gamma}}^{\ell^{ k}}(Z,Z)\|_{\ell^{ k}}=O_{X_{\Gamma}}\big(k^{3g(g+1)/4}\big),
\end{align*}
%%%%%%%%%%%%%%%%%%%%%%%%%%%%%%%%%%%%%%%%%%%%%%%%%%%%%%%%%%%%%%%%%%%%%%%%%%%%%%%%%%%%%%%%%%
where the implied constant depends only on $X_{\Gamma}$.
%%%%%%%%%%%%%%%%%%%%%%%%%%%%%%%%%%%%%%%%%%%%%%%%%%%%%%%%%%%%%%%%%%%%%%%%%%%%%%%%%%%%%%%%%%
\end{rem}
%%%%%%%%%%%%%%%%%%%%%%%%%%%%%%%%%%%%%%%%%%%%%%%%%%%%%%%%%%%%%%%%%%%%%%%%%%%%%%%%%%%%%%%%%%
%%%%%%%%%%%%%%%%%%%%%%%%%%%%%%%%%%%%%%%%%%%%%%%%%%%%%%%%%%%%%%%%%%%%%%%%%%%%%%%%%%%%%%%%%%

\vspace{0.2cm}
\subsection{Strategy of the proof}\label{sec-1.3}
%%%%%%%%%%%%%%%%%%%%%%%%%%%%%%%%%%%%%%%%%%%%%%%%%%%%%%%%%%%%%%%%%%%%%%%%%%%%%%%%%%%%%%%%%%
Our strategy is similar to the techniques employed in \cite{anil1}, \cite{anil2} and \cite{anil3}. Let $S_{k(g+1)}(\Gamma)$ denote the complex vector space of Siegel cusp forms of weight $k(g+1)$, which is equipped with the Petersson inner product and Petersson norm, which is denoted by $\|\cdot \|_{\mathrm{pet}}$. Let $\mathcal{B}_{\Gamma}^{k(g+1)}$ denote the Bergman kernel for the space $S_{k(g+1)}(\Gamma)$, which can be explicitly expressed as an infinite series (see section \ref{sec-2.2} for the precise expression for the series).
%%%%%%%%%%%%%%%%%%%%%%%%%%%%%%%%%%%%%%%%%%%%%%%%%%%%%%%%%%%%%%%%%%%%%%%%%%%%%%%%%%%%%%%%%%

Recalling the isometry $H^{0}_{L^{2}}(X_{\Gamma},\ell^{k})\simeq \mathcal{S}_{k(g+1)}(\Gamma)$, for any $Z,W\in X_{\Gamma}$, we have the following equality
\begin{align*}
\|\mathcal{B}_{X_{\Gamma}}^{\ell^{ k}}(Z,W)\|_{\ell^{ k}}=\|\mathcal{B}_{\Gamma}^{k(g+1)}(Z,W)\|_{\mathrm{pet}}.
\end{align*}
%%%%%%%%%%%%%%%%%%%%%%%%%%%%%%%%%%%%%%%%%%%%%%%%%%%%%%%%%%%%%%%%%%%%%%%%%%%%%%%%%%%%%%%%%%

Using the infinite series expansion for $\mathcal{B}_{\Gamma}^{k(g+1)}$, we estimate $\mathcal{B}_{X_{\Gamma}}^{\ell^{ k}}$. Furthermore, to estimate the infinite series expansion for $\mathcal{B}_{\Gamma}^{k(g+1)}$, we derive a formula for counting elements of $\Gamma$. This is a generalization of the formula proved by Jorgenson and Lundelius in \cite{jl}, from the setting of Riemann surfaces to Siegel modular varieties, or stated equivalently, from the setting of genus $g=1$ to genus $g\geq 2$.
%%%%%%%%%%%%%%%%%%%%%%%%%%%%%%%%%%%%%%%%%%%%%%%%%%%%%%%%%%%%%%%%%%%%%%%%%%%%%%%%%%%%%%%%%%

Hereafter, we refer to the counting formula as Jorgenson-Lundelius type estimate, as is proved in section \ref{sec-3.1} as Theorem \ref{thm4}.
%%%%%%%%%%%%%%%%%%%%%%%%%%%%%%%%%%%%%%%%%%%%%%%%%%%%%%%%%%%%%%%%%%%%%%%%%%%%%%%%%%%%%%%%%%
%%%%%%%%%%%%%%%%%%%%%%%%%%%%%%%%%%%%%%%%%%%%%%%%%%%%%%%%%%%%%%%%%%%%%%%%%%%%%%%%%%%%%%%%%%
%%%%%%%%%%%%%%%%%%%%%%%%%%%%%%%%%%%%%%%%%%%%%%%%%%%%%%%%%%%%%%%%%%%%%%%%%%%%%%%%%%%%%%%%%%
%%%%%%%%%%%%%%%%%%%%%%%%%%%%%%%%%%%%%%%%%%%%%%%%%%%%%%%%%%%%%%%%%%%%%%%%%%%%%%%%%%%%%%%%%%

\vspace{0.2cm}
\section{Background material}\label{sec-2}
%%%%%%%%%%%%%%%%%%%%%%%%%%%%%%%%%%%%%%%%%%%%%%%%%%%%%%%%%%%%%%%%%%%%%%%%%%%%%%%%%%%%%%%%%%
In this section we set up the notation required to prove Theorem \ref{mainthm1} and Theorem \ref{mainthm2}. Furthermore, we recall important results from literature which will assist us in proving our main results.
%%%%%%%%%%%%%%%%%%%%%%%%%%%%%%%%%%%%%%%%%%%%%%%%%%%%%%%%%%%%%%%%%%%%%%%%%%%%%%%%%%%%%%%%%%%
%%%%%%%%%%%%%%%%%%%%%%%%%%%%%%%%%%%%%%%%%%%%%%%%%%%%%%%%%%%%%%%%%%%%%%%%%%%%%%%%%%%%%%%%%%%

\vspace{0.2cm}
\subsection{Siegel modular variety}\label{sec-2.1}
%%%%%%%%%%%%%%%%%%%%%%%%%%%%%%%%%%%%%%%%%%%%%%%%%%%%%%%%%%%%%%%%%%%%%%%%%%%%%%%%%%%%%%%%%%%
Let $\mathrm{Sym}_{g}(\mathbb{R})$ denote the set of $g\times g$ symmetric matrices, and let
\begin{align*}
    \mathbb{H}_{g}:=\big\lbrace Z=X+iY\big|\,X,Y\in \mathrm{Sym}_{g}(\mathbb{R}),\,\mathrm{det}(Y)>0\big\rbrace
\end{align*}
%%%%%%%%%%%%%%%%%%%%%%%%%%%%%%%%%%%%%%%%%%%%%%%%%%%%%%%%%%%%%%%%%%%%%%%%%%%%%%%%%%%%%%%%%%%
denote the Siegel upper space of genus $g\geq 2$.  The natural metric on $\mathbb{H}_{g}$ is the Siegel metric, which is the natural metric compatible with its complex structure.
Let $\mu_{\mathrm{S}^{\mathrm{Vol}_{\mathrm{S}}}}$ denote the associated volume form, and at any $Z=(z_{jk}=x_{jk}+iy_{jk})_{1\leq j,k\leq g}\in\mathbb{H}_{g}$, the Siegel metric is given
by the following formula
\begin{align}\label{hypvol}
    \mu_{\mathrm{S}}^{\mathrm{Vol}_{\mathrm{S}}}(Z):=\frac{\displaystyle\bigwedge_{1\leq j\leq k\leq g}\dd x_{jk}\wedge \dd y_{jk}}{\mathrm{det}(Y)^{g+1}}.
\end{align}
%%%%%%%%%%%%%%%%%%%%%%%%%%%%%%%%%%%%%%%%%%%%%%%%%%%%%%%%%%%%%%%%%%%%%%%%%%%%%%%%%%%%%%%%%%%

The sympletic group $\mathrm{Sp}(2g,\mathbb{R})$ act on the Siegel upper half space $\mathbb{H}_{g}$ via fractional linear transformations. A discrete subgroup $\Gamma\subset \mathrm{Sp}(2g,\mathbb{R})$ is called cocompact, if the quotient space $X_{\Gamma}:=\Gamma\backslash\mathbb{H}_{g}$ admits the structure of a complex manifold. Let $\Gamma_{0}:=\mathrm{Sp}(2g,\mathbb{Z})$ denote the full modular group. A discrete subgroup $\Gamma\subset \mathrm{Sp}(2g,\mathbb{R})$ is called an arithmetic subgroup, if the index of the subgroup $\Gamma\cap\Gamma_{0}$ is finite both in $\Gamma_{0}$ and $\Gamma$.
%%%%%%%%%%%%%%%%%%%%%%%%%%%%%%%%%%%%%%%%%%%%%%%%%%%%%%%%%%%%%%%%%%%%%%%%%%%%%%%%%%%%%%%%%%%

For the rest of the article, $\Gamma$ denotes either a cocompact subgroup or an arithmetic subgroup without elliptic fixed points. The quotient space $X_{\Gamma}$ is a complex manifold of dimension $g(g+1)/2$, which is compact when $\Gamma$ is cocompact, and is non compact, when $\Gamma$ is an arithmetic subgroup. The Siegel metric descends to define a metric on $X_{\Gamma}$, and locally, it is given by the same formula as the one described in equation \eqref{hypvol}.
%%%%%%%%%%%%%%%%%%%%%%%%%%%%%%%%%%%%%%%%%%%%%%%%%%%%%%%%%%%%%%%%%%%%%%%%%%%%%%%%%%%%%%%%%%%

The geodesic distance between any two given points $Z,W\in X_{\Gamma}$ (we identify $X_{\Gamma}$ with its universal cover $\mathbb{H}_{g}$), which is determined by the Siegel metric, is again denoted by $d_{\mathrm{S}}(Z,W)$, and is again given by the following formula
\begin{align}\label{hyp-dist}
d_{\mathrm{S}}(Z,W)=\sqrt{2}\bigg(\sum_{j=1}^{g}\log^{2}\bigg( \frac{1+\sqrt{\rho_{j}(Z,W)}}{1-\sqrt{\rho_{j}(Z,W)}}\bigg)\bigg)^{1/2},
\end{align}
%%%%%%%%%%%%%%%%%%%%%%%%%%%%%%%%%%%%%%%%%%%%%%%%%%%%%%%%%%%%%%%%%%%%%%%%%%%%%%%%%%%%%%%%%%%
where $\lbrace \rho_1(Z,W),\ldots, \rho_{g}(Z,W)\rbrace$ denotes the set of eigenvalues of the cross-ratio matrix
\begin{align*}
\rho(Z,W):=(Z-W)(\overline{Z}-W)^{-1}(\overline{Z}-\overline{W})(Z-\overline{W})^{-1}.
\end{align*}
%%%%%%%%%%%%%%%%%%%%%%%%%%%%%%%%%%%%%%%%%%%%%%%%%%%%%%%%%%%%%%%%%%%%%%%%%%%%%%%%%%%%%%%%%%%

We now define the Dirichlet Fundamental domain of $X_{\Gamma}$ centered at a given $z\in X_{\Gamma}$ (we identify $X_{\Gamma}$ with its universal cover $\mathbb{H}_{g}$
\begin{align}\label{dir}
\mathcal{F}_{\Gamma,Z}:=\big\lbrace W\in \mathbb{H}_{g}\big|\,d_{\mathrm{S}}(Z, W)<d_{\mathrm{S}}(Z,\gamma W),\,\mathrm{for\,\,all}\,\,\gamma\in\Gamma\backslash\lbrace\mathrm{Id}\rbrace\big\rbrace.
\end{align}
%%%%%%%%%%%%%%%%%%%%%%%%%%%%%%%%%%%%%%%%%%%%%%%%%%%%%%%%%%%%%%%%%%%%%%%%%%%%%%%%%%%%%%%%%%%

When $\Gamma$ is cocompact, we define the injectivity radius as
\begin{align}\label{irad-1}
r_{\Gamma}:=\frac{1}{2}\inf\big\lbrace d_{\mathrm{S}}(Z,\gamma Z)\big|\,Z\in \mathbb{H}_{g},\,\,\mathrm{for\,\,all}\,\,\gamma\in\Gamma\backslash\lbrace\mathrm{Id}\rbrace\big\rbrace.
\end{align}
%%%%%%%%%%%%%%%%%%%%%%%%%%%%%%%%%%%%%%%%%%%%%%%%%%%%%%%%%%%%%%%%%%%%%%%%%%%%%%%%%%%%%%%%%%%
Observe that $2r_{\Gamma}$ is the length of the shortest geodesic on $X_{\Gamma}$.  
%%%%%%%%%%%%%%%%%%%%%%%%%%%%%%%%%%%%%%%%%%%%%%%%%%%%%%%%%%%%%%%%%%%%%%%%%%%%%%%%%%%%%%%%%%%
%%%%%%%%%%%%%%%%%%%%%%%%%%%%%%%%%%%%%%%%%%%%%%%%%%%%%%%%%%%%%%%%%%%%%%%%%%%%%%%%%%%%%%%%%%%

\vspace{0.2cm}
\subsection{$\Gamma$ is an arithmetic subgroup}\label{sec-2.2}
%%%%%%%%%%%%%%%%%%%%%%%%%%%%%%%%%%%%%%%%%%%%%%%%%%%%%%%%%%%%%%%%%%%%%%%%%%%%%%%%%%%%%%%%%%%
Without loss of generality, we assume that $X_{\Gamma}$ admits only one $\Gamma$ equivalent class of rational boundary component, and let
\begin{align*}
&\overline{X_{\Gamma}}:=\bigsqcup_{j=1}^{g}\Gamma_{j}\backslash\mathbb{H}_{j}\sqcup\lbrace i\infty\rbrace, \\
&\mathrm{where}\,\,\Gamma_{j}\subset \mathrm{Sp}(2j,\mathbb{Z}), \,\,\mathrm{for\,\,all}\,\,1\leq j\leq g-1.
\end{align*}
%%%%%%%%%%%%%%%%%%%%%%%%%%%%%%%%%%%%%%%%%%%%%%%%%%%%%%%%%%%%%%%%%%%%%%%%%%%%%%%%%%%%%%%%%%%
Furthermore, let
\begin{align}\label{bdy}
\mathcal{D}_{\mathrm{bdy}}:=\overline{X_{\Gamma}}\backslash X_{\Gamma}
\end{align}
%%%%%%%%%%%%%%%%%%%%%%%%%%%%%%%%%%%%%%%%%%%%%%%%%%%%%%%%%%%%%%%%%%%%%%%%%%%%%%%%%%%%%%%%%%%
denote the boundary divisor of $\overline{X_{\Gamma}}$.
%%%%%%%%%%%%%%%%%%%%%%%%%%%%%%%%%%%%%%%%%%%%%%%%%%%%%%%%%%%%%%%%%%%%%%%%%%%%%%%%%%%%%%%%%%%

Without loss of generality, we assume that $\Gamma$ a fundamental domain $\mathcal{F}_{\Gamma}^{\mathrm{Siegel}}$, which is known as the Siegel fundamental domain, and satisfies the following properties:
\begin{enumerate}
\item[(i)]
For every $Z\in\mathcal{F}_{\Gamma}^{\mathrm{Siegel}}$, and $\big(\begin{smallmatrix} A &B\\C& D\end{smallmatrix}\big)\in\Gamma$, we have 
$\mathrm{det}(CZ+D)\geq 1$.
%%%%%%%%%%%%%%%%%%%%%%%%%%%%%%%%%%%%%%%%%%%%%%%%%%%%%%%%%%%%%%%%%%%%%%%%%%%%%%%%%%%%%%%%%%%
\item[(ii)] For every $Z=X+iY\in\mathcal{F}_{\Gamma}^{\mathrm{Siegel}}$, $Y=(y_{jk})_{1\leq j,k\leq g}$ is Minkowski reduced, i.e.,  
$y_{k,k+1}\geq  0$ for $(1 \leq  k\leq g- 1)$, and for all primitive vectors $h\in\mathbb{Z}^{g}$, we have $h^tYh\geq  y_{k,k}$ for $(1 \leq  k\leq n)$, where a vector $h = (h_1,\ldots h_g)^{t} \in \mathbb{Z}^{g}$ is called a primitive vector, if for  $1 \leq k \leq g$, we have $\mathrm{gcd}(h_k ,\ldots,h_g) = 1$.
%%%%%%%%%%%%%%%%%%%%%%%%%%%%%%%%%%%%%%%%%%%%%%%%%%%%%%%%%%%%%%%%%%%%%%%%%%%%%%%%%%%%%%%%%%%
\item[(iii)] For every $Z=X+iY\in\mathcal{F}_{\Gamma}^{\mathrm{Siegel}}$ with $X:=(x_{jk})_{1\leq j,k\leq g}$, we have $|x_{jk}|<1/2$.
\end{enumerate}
%%%%%%%%%%%%%%%%%%%%%%%%%%%%%%%%%%%%%%%%%%%%%%%%%%%%%%%%%%%%%%%%%%%%%%%%%%%%%%%%%%%%%%%%%%%

Let $\Gamma_{\infty}\subset \Gamma$ denote the stabilizer of the boundary component $\mathcal{D}_{\mathrm{bdy}}$, which is given by the following equation
\begin{align}\label{gamma-infty0}
&\Gamma_{\infty}:=\bigcup_{j=0}^{g-1}\Gamma_{\infty}^{j},\notag                                                                                                               \\
&\mathrm{where}\,\,\Gamma_{\infty}^{0}:=\bigg\lbrace \begin{pmatrix} \mathrm{Id}_{g}&S\\ 0 &\mathrm{Id}_{g}\end{pmatrix}\bigg|\,S\in\mathrm{Sym}_{g}(\mathbb{Z})\bigg\rbrace,
\end{align}
%%%%%%%%%%%%%%%%%%%%%%%%%%%%%%%%%%%%%%%%%%%%%%%%%%%%%%%%%%%%%%%%%%%%%%%%%%%%%%%%%%%%%%%%%%%
and for $1\leq j\leq g-1$, we have
\begin{align}\label{gamma-inftyj}
&\Gamma_{\infty}^{j}:= \bigg\lbrace \begin{pmatrix} A&AS\\ 0 &A^{-t}\end{pmatrix}\bigg|\,A=\begin{pmatrix} \mathrm{Id}_{j}&0\\ L &\mathrm{Id}_{g-j}\end{pmatrix},\,\,S=\begin{pmatrix} 0&H^{t}\\ H &M\end{pmatrix}\bigg\rbrace,\notag \\[0.1cm]
&\mathrm{where}\,\,L,\,H\in\mathbb{Z}^{(g-j)\times j},\,M\in\mathrm{Sym}_{g-j}(\mathbb{Z}).
\end{align}
%%%%%%%%%%%%%%%%%%%%%%%%%%%%%%%%%%%%%%%%%%%%%%%%%%%%%%%%%%%%%%%%%%%%%%%%%%%%%%%%%%%%%%%%%%%

We define the injectivity radius of $X_{\Gamma}$ as
\begin{align}\label{irad-2}
 r_{\Gamma}:=\frac{1}{2}\inf\big\lbrace d_{\mathrm{S}}(Z,\gamma Z)\big|\,Z\in \mathcal{F}_{\Gamma},\,\,\mathrm{for\,\,all}\,\,\gamma\in\Gamma\backslash\Gamma_{\infty}\big\rbrace,
\end{align}
%%%%%%%%%%%%%%%%%%%%%%%%%%%%%%%%%%%%%%%%%%%%%%%%%%%%%%%%%%%%%%%%%%%%%%%%%%%%%%%%%%%%%%%%%%%
and  $\mathcal{F}_{\Gamma}$ denotes a fixed fundamental domain of $X_{\Gamma}$.
%%%%%%%%%%%%%%%%%%%%%%%%%%%%%%%%%%%%%%%%%%%%%%%%%%%%%%%%%%%%%%%%%%%%%%%%%%%%%%%%%%%%%%%%%%%
%%%%%%%%%%%%%%%%%%%%%%%%%%%%%%%%%%%%%%%%%%%%%%%%%%%%%%%%%%%%%%%%%%%%%%%%%%%%%%%%%%%%%%%%%%%

\vspace{0.2cm}
\subsection{Bergman kernel}\label{sec-2.3}
%%%%%%%%%%%%%%%%%%%%%%%%%%%%%%%%%%%%%%%%%%%%%%%%%%%%%%%%%%%%%%%%%%%%%%%%%%%%%%%%%%%%%%%%%%%
Let $\Omega_{X_{\Gamma}}$ denote the cotangent bundle of $X_{\Gamma}$, and let $\ell:=\mathrm{det}(\Omega_{X_{\Gamma}})$ denote the determinant line bundle of the cotangent bundle $\Omega_{X_{\Gamma}}$, which is also known as the canonical line bundle. also known as the canonical line bundle. For $k\geq 1$,
let $H^{0}(X_{\Gamma},\ell^{\otimes k})$ denote the space of global holomorphic sections, which is equipped with a point-wise norm $\|\cdot\|_{\ell^{k}}$ and an $L^{2}$ norm
$\|\cdot\|_{\ell^k,L^{2}}$, which is induced by an $L^2$ inner product.
%%%%%%%%%%%%%%%%%%%%%%%%%%%%%%%%%%%%%%%%%%%%%%%%%%%%%%%%%%%%%%%%%%%%%%%%%%%%%%%%%%%%%%%%%%%

Let $\mathcal{B}^{\ell^{k}}_{X_{\Gamma}}$ denote the Bergman kernel associated to the vector subspace
\begin{align*}
H^{0}_{L^{2}}(X_{\Gamma},\ell^{\otimes k})=H^{0}\big(\overline{X_{\Gamma}},\overline{\ell}^{\otimes k}\otimes O_{\overline{X_{\Gamma}}}(-\mathcal{D}_{\mathrm{bdy}})\big)\subset  H^{0}(X_{\Gamma},\ell^{\otimes k}),
\end{align*}
%%%%%%%%%%%%%%%%%%%%%%%%%%%%%%%%%%%%%%%%%%%%%%%%%%%%%%%%%%%%%%%%%%%%%%%%%%%%%%%%%%%%%%%%%%%
where $\overline{\ell}^{\otimes k}$ denotes the extension of $\ell^{\otimes k}$ to $\overline{X_{\Gamma}}$.
%%%%%%%%%%%%%%%%%%%%%%%%%%%%%%%%%%%%%%%%%%%%%%%%%%%%%%%%%%%%%%%%%%%%%%%%%%%%%%%%%%%%%%%%%%%

We now recall the isometry $H^{0}_{L^{2}}(X_{\Gamma},\ell^{\otimes k})\simeq \mathcal{S}_{k(g+1)}(\Gamma)$, where $\mathcal{S}_{k(g+1)}(\Gamma)$ denotes the complex vector
space of weight $k(g+1)$ cusp forms, which is equipped with a point-wise Petersson norm $\|\cdot\|_{\mathrm{pet}}$, and a Petersson inner product $\langle \cdot,\cdot\rangle_{\mathrm{pet}}$.
%%%%%%%%%%%%%%%%%%%%%%%%%%%%%%%%%%%%%%%%%%%%%%%%%%%%%%%%%%%%%%%%%%%%%%%%%%%%%%%%%%%%%%%%%%%

For any $f\in\mathcal{S}_{k(g+1)}(\Gamma)$, at any $Z=X+iY\in \mathbb{H}_{g}$, the the Petersson norm of $f$ is given by the following formula
\begin{align}\label{petnorm}
\|f(Z)\|_{\mathrm{pet}}^{2}:=\mathrm{det}(Y)^{k(g+1)}\big|f(Z)\big|^{2}.
\end{align}
%%%%%%%%%%%%%%%%%%%%%%%%%%%%%%%%%%%%%%%%%%%%%%%%%%%%%%%%%%%%%%%%%%%%%%%%%%%%%%%%%%%%%%%%%%%%
Observe that the Petersson norm $\|\cdot\|_{\mathrm{pet}}$ is invariant with respect to action of $\Gamma$, and hence, defines a function on $X_{\Gamma}$.
%%%%%%%%%%%%%%%%%%%%%%%%%%%%%%%%%%%%%%%%%%%%%%%%%%%%%%%%%%%%%%%%%%%%%%%%%%%%%%%%%%%%%%%%%%%%

Let $\lbrace f_{1},\ldots,f_{d_{k}} \rbrace$ denote an orthonormal basis of $\mathcal{S}_{k(g+1)}(\Gamma)$, with respect to the Petersson inner product, where $d_{k}$ denotes the dimension of
$\mathcal{S}_{k(g+1)}(\Gamma)$ as a complex vector space. For any $Z,W\in\mathbb{H}_{g}$, the Bergman kernel associated to $\mathcal{S}_{k(g+1)}(\Gamma)$ is given by the following formula
\begin{align*}
\mathcal{B}_{\Gamma}^{k(g+1)}(Z,W)=\sum_{j=1}^{d_{k}}f_{j}(Z)\overline{f_{j}(W)}.
\end{align*}
%%%%%%%%%%%%%%%%%%%%%%%%%%%%%%%%%%%%%%%%%%%%%%%%%%%%%%%%%%%%%%%%%%%%%%%%%%%%%%%%%%%%%%%%%%%
At any $Z=X+iY,\,W=U+iV\in \mathbb{H}_{g}$, the Petersson norm is given by the following formula
\begin{align}\label{bk-petnorm}
\|\mathcal{B}_{\Gamma}^{k(g+1)}(Z,W)\|_{\mathrm{pet}}=\mathrm{det}(YV)^{k(g+1)/2}\big| \mathcal{B}_{\Gamma}^{k(g+1)}(Z,W)\big|
\end{align}
%%%%%%%%%%%%%%%%%%%%%%%%%%%%%%%%%%%%%%%%%%%%%%%%%%%%%%%%%%%%%%%%%%%%%%%%%%%%%%%%%%%%%%%%%%%
By Reisz representation theorem, the definition of the Bergman kernel is independent of the choice of orthonormal basis for $\mathcal{S}_{k(g+1)}(\Gamma)$. For any $Z=X+iY,W=U+iV\in \mathbb{H}_{g}$, the Bergman kernel $\mathcal{B}_{\Gamma}^{k(g+1)}$ is also described by the following formula
\begin{align}\label{bkseries}
& \mathcal{B}_{\Gamma}^{k(g+1)}(Z,W)=\sum_{\gamma={\tiny{(\begin{smallmatrix} A &B\\C& D\end{smallmatrix})}\in\Gamma}}\frac{(4^{g})^{k(g+1)/2} C_{g(k+1)} }{(\mathrm{det}{(Z-\overline{\gamma W} )})^{k(g+1)}\cdot(\mathrm{det}(\overline{CZ+D}))^{k(g+1)}},\notag \\
& \mathrm{where}\,\,C_{k(g+1)}=O_{X_{\Gamma}}\big(k^{g(g+1)/2}\big),
\end{align}
%%%%%%%%%%%%%%%%%%%%%%%%%%%%%%%%%%%%%%%%%%%%%%%%%%%%%%%%%%%%%%%%%%%%%%%%%%%%%%%%%%%%%%%%%%%
where the implied constant depends only on $X_{\Gamma}$.
%%%%%%%%%%%%%%%%%%%%%%%%%%%%%%%%%%%%%%%%%%%%%%%%%%%%%%%%%%%%%%%%%%%%%%%%%%%%%%%%%%%%%%%%%%%

We now derive an estimate for the Bergman kernel, which is well known for experts. However, for the sake of completion, and for the convenience of the reader, we reprove the estimate
here.
%%%%%%%%%%%%%%%%%%%%%%%%%%%%%%%%%%%%%%%%%%%%%%%%%%%%%%%%%%%%%%%%%%%%%%%%%%%%%%%%%%%%%%%%%%%

Let $Z=X+iY,\, W =U+iV \in\mathbb{H}_{g}$, and let $\gamma \in \mathrm{Sp}(2g, \mathbb{R})$. Let $\lbrace \rho_1(Z,W),\ldots,\rho_{g}(Z,W) \rbrace$ be the eigenvalues of the cross-ratio matrix $\rho(Z,W)$. Observe that $\rho(Z,W)=\rho(\gamma Z,\gamma W )$. Furthermore,  since the action of  $ \mathrm{Sp}(2g, \mathbb{R})$ on $\mathbb{H}_{g}$ is transitive, there exists a  $\sigma \in \mathrm{Sp}(2g, \mathbb{R})$ such that $\sigma Z = i\mathrm{Id}_{g}$. For example the following matrix
\begin{align*}
\sigma = \begin{pmatrix}
Y^{-1/2} & -Y^{-1/2}X \\0 & Y^{1/2}
\end{pmatrix},
\end{align*}
%%%%%%%%%%%%%%%%%%%%%%%%%%%%%%%%%%%%%%%%%%%%%%%%%%%%%%%%%%%%%%%%%%%%%%%%%%%%%%%%%%%%%%%%%%%
satisfies the equation $\sigma Z = i\mathrm{Id}_{g}$.
%%%%%%%%%%%%%%%%%%%%%%%%%%%%%%%%%%%%%%%%%%%%%%%%%%%%%%%%%%%%%%%%%%%%%%%%%%%%%%%%%%%%%%%%%%%

For  $\sigma$ as above, we compute
\begin{align}\label{eq:sigmaW-definition}
\sigma W = Y^{-1/2}(W-X)Y^{-1/2}
\end{align}
%%%%%%%%%%%%%%%%%%%%%%%%%%%%%%%%%%%%%%%%%%%%%%%%%%%%%%%%%%%%%%%%%%%%%%%%%%%%%%%%%%%%%%%%%%%

For brevity of notation, we now represent $\rho(Z,W)=\rho(\gamma Z,\gamma W)$ by $\rho$, and the set of eigenvalues of $\rho$ by $\lbrace \rho_{1},\ldots,\rho_{g}\rbrace$.  For $1\leq j\leq g$, put $r_j = \tanh^{-1}(\sqrt{\rho_j})$. Then, we have
\begin{align}\label{bkseries-1}
d_{\mathrm{S}}(Z,W) = d_{\mathrm{S}}(\sigma Z,\sigma W) = \bigg( 8 \sum_{j=1}^{g} r_j^2 \bigg)^{1/2}.
\end{align}
%%%%%%%%%%%%%%%%%%%%%%%%%%%%%%%%%%%%%%%%%%%%%%%%%%%%%%%%%%%%%%%%%%%%%%%%%%%%%%%%%%%%%%%%%%%
From the definition of cross-ratio matrix $\rho(Z,W)$, we have
\begin{align}\label{bkseries-2}
\rho(\sigma Z,\sigma W) =  \rho(i\mathrm{Id}_{g},\sigma W) = (i\mathrm{Id}_{g} - \sigma W)(-i\mathrm{Id}_{g} - \sigma W)^{-1}(-i\mathrm{Id}_{g} - \overline{\sigma W})(i\mathrm{Id}_{g} - \overline{\sigma W})^{-1}.
\end{align}
%%%%%%%%%%%%%%%%%%%%%%%%%%%%%%%%%%%%%%%%%%%%%%%%%%%%%%%%%%%%%%%%%%%%%%%%%%%%%%%%%%%%%%%%%%%
Combining equations \eqref{bkseries-1} and \eqref{bkseries-2}, we deduce that
\begin{align}\label{bkseries-3}
 \mathrm{det}{(\mathrm{Id}_{g} - \rho(i\mathrm{Id}_{g},\sigma W))}=\mathrm{det}{( \mathrm{Id}_{g}- \rho(Z,W))}= \prod_{j=1}^{g}(1 - \tanh^2{r_j})=\prod_{j=1}^{g}\frac{1}{\cosh^2(r_j)}.
\end{align}
%%%%%%%%%%%%%%%%%%%%%%%%%%%%%%%%%%%%%%%%%%%%%%%%%%%%%%%%%%%%%%%%%%%%%%%%%%%%%%%%%%%%%%%%%%%
Using equation \eqref{bkseries-2}, we now compute
\begin{align}\label{bkseries-4}
\mathrm{det}{(\mathrm{Id}_{g} - \rho(i\mathrm{Id}_{g},\sigma W))} = \frac{\mathrm{det}{\big( (i\mathrm{Id}_{g} + \sigma W)(i\mathrm{Id}_{g} - \overline{\sigma W})-(i\mathrm{Id}_{g} - \sigma W)(i\mathrm{Id}_{g} + \overline{\sigma W}) \big)}}{\mathrm{det}{(i\mathrm{Id}_{g} + \sigma W)}\mathrm{det}{(i\mathrm{Id}_{g} - \overline{\sigma W})}}=\notag \\[0.1cm]\frac{\mathrm{det}{\big(2i(\sigma W - \overline{\sigma W})\big)}}{\mathrm{det}{(i\mathrm{Id}_{g} + \sigma W)}\mathrm{det}{(i\mathrm{Id}_{g} - \overline{\sigma W})}}.
\end{align}
%%%%%%%%%%%%%%%%%%%%%%%%%%%%%%%%%%%%%%%%%%%%%%%%%%%%%%%%%%%%%%%%%%%%%%%%%%%%%%%%%%%%%%%%%%%
Furthermore, from equation \eqref{eq:sigmaW-definition}, we also have
\begin{align}\label{bkseries-5}
& \mathrm{det}{(i\mathrm{Id}_{g} + \sigma W)} = \mathrm{det}{(Y^{-1/2}(W - \overline{Z})Y^{-1/2})}, \notag \\
& \mathrm{det}{(i\mathrm{Id}_{g} - \sigma W)} = \mathrm{det}{(Y^{-1/2}(Z - \overline{W})Y^{-1/2})}.
\end{align}
%%%%%%%%%%%%%%%%%%%%%%%%%%%%%%%%%%%%%%%%%%%%%%%%%%%%%%%%%%%%%%%%%%%%%%%%%%%%%%%%%%%%%%%%%%%
Combining equations \eqref{bkseries-4}, and \eqref{bkseries-5}, we derive
\begin{align}\label{bkseries-6}
\mathrm{det}{(\mathrm{Id}_{g} - \rho(i\mathrm{Id}_{g},\sigma W))}=\frac{\mathrm{det}{(2i(\sigma W - \overline{\sigma W}))}}{(-1)^{g}\big|\mathrm{det}{(Y^{-1/2}(Z - \overline{W})Y^{-1/2})}\big|^{2}}
 =\notag \\ \frac{4^g\mathrm{det}{( Y^{-1/2}VY^{-1/2} )}}{\left|\mathrm{det}{(Y^{-1/2}(Z- \overline{W})Y^{-1/2})}\right|^2}.
\end{align}
%%%%%%%%%%%%%%%%%%%%%%%%%%%%%%%%%%%%%%%%%%%%%%%%%%%%%%%%%%%%%%%%%%%%%%%%%%%%%%%%%%%%%%%%%%%
Combining equations \eqref{bkseries-3} and \eqref{bkseries-6}, we arrive at the following equality
\begin{align}\label{bkseries-7}
\frac{ \mathrm{det}{(4YV)}}{\left|\mathrm{det}{(Z- \overline{W})}\right|^2}=\prod_{j=1}^{g}\frac{1}{\cosh^2{r_j}}.
\end{align}
%%%%%%%%%%%%%%%%%%%%%%%%%%%%%%%%%%%%%%%%%%%%%%%%%%%%%%%%%%%%%%%%%%%%%%%%%%%%%%%%%%%%%%%%%%%
Hence, from equations \eqref{bk-petnorm}, \eqref{bkseries}, \eqref{bkseries-7}, for $\gamma={\tiny{(\begin{smallmatrix} A &B\\C& D\end{smallmatrix}})}\in\Gamma$, using the fact that
\begin{align*}
\mathrm{Im}(\gamma W)=\frac{V}{\big|\mathrm{det}(CW+D)\big|^{2}},
\end{align*}
%%%%%%%%%%%%%%%%%%%%%%%%%%%%%%%%%%%%%%%%%%%%%%%%%%%%%%%%%%%%%%%%%%%%%%%%%%%%%%%%%%%%%%%%%%%
we infer the following estimate for the Bergman kernel
\begin{align}\label{bk-est}
&\|\mathcal{B}_{\Gamma}^{k(g+1)}(Z,W)\|_{\mathrm{pet}}=\bigg|\sum_{\gamma={\tiny{(\begin{smallmatrix} A &B\\C& D\end{smallmatrix})}\in\Gamma}}\frac{C_{g(k+1)}\big( \mathrm{det}(4YV)\big)^{k(g+1)/2}}{\mathrm{det}{(Z-\overline{\gamma W}) }^{k(g+1)}\cdot\mathrm{det}(\overline{CZ+D})^{k(g+1)}}\bigg|\leq\notag \\ &
 \sum_{\gamma={\tiny{(\begin{smallmatrix} A &B\\C& D\end{smallmatrix})}\in\Gamma}}\frac{ C_{g(k+1)}\big( \mathrm{det}(4YV)\big)^{k(g+1)/2}}{\big|\mathrm{det}{(Z-\overline{\gamma W} )}\big|^{k(g+1)}\cdot\big|\mathrm{det}(\overline{CZ+D})\big|^{k(g+1)}}\leq\sum_{\gamma\in\Gamma}\prod_{j=1}^{g}\frac{C_{g(k+1)}}{\cosh^{k(g+1)}(r_j(Z,\gamma W))}.
\end{align}
%%%%%%%%%%%%%%%%%%%%%%%%%%%%%%%%%%%%%%%%%%%%%%%%%%%%%%%%%%%%%%%%%%%%%%%%%%%%%%%%%%%%%%%%%%%
Finally, from estimate \eqref{bk-est}, and using isometry  $H^{0}_{L^{2}}(X_{\Gamma},\ell^{\otimes k})\simeq \mathcal{S}_{k(g+1)}(\Gamma)$, we deduce that
\begin{align}\label{bk-est2}
\|\mathcal{B}_{X_{\Gamma}}^{\ell^{k}}(Z,W)\|_{\ell^k}=\|\mathcal{B}_{\Gamma}^{k(g+1)}(Z,W)\|_{\mathrm{pet}}\leq \sum_{\gamma\in\Gamma}\prod_{j=1}^{g}\frac{C_{g(k+1)}}{\cosh^{k(g+1)}(r_j(Z,\gamma W))}.
\end{align}
%%%%%%%%%%%%%%%%%%%%%%%%%%%%%%%%%%%%%%%%%%%%%%%%%%%%%%%%%%%%%%%%%%%%%%%%%%%%%%%%%%%%%%%%%%%
%%%%%%%%%%%%%%%%%%%%%%%%%%%%%%%%%%%%%%%%%%%%%%%%%%%%%%%%%%%%%%%%%%%%%%%%%%%%%%%%%%%%%%%%%%%

\vspace{0.2cm}
\subsection{Siegel metric in polar coordinates}\label{sec-2.4}
%%%%%%%%%%%%%%%%%%%%%%%%%%%%%%%%%%%%%%%%%%%%%%%%%%%%%%%%%%%%%%%%%%%%%%%%%%%%%%%%%%%%%%%%%%%
We now describe the Siegel metric in polar coordinates. Any $Z\in \mathbb{H}_{g}$, can be expressed as
\begin{align*}
Z:=q\,\exp(2R)i,\,\,\mathrm{where} \,R=\big(\begin{smallmatrix}D_{r}&0\\0&-D_{r}  \end{smallmatrix}\big),\,\,\mathrm{and}\,\,D_{r}= \Bigg(\begin{smallmatrix}
r_{1} & & \\ & \ddots & \\ & & r_{g}\end{smallmatrix}\Bigg)
\,\,\mathrm{is\,\,a\,\,diagonal\,\,matrix}, \\[0.1cm]\mathrm{and}\,\,q\in K:=\mathrm{Sp}(2g,\mathbb{R})\cap\mathrm{O}(2g,\mathbb{R}).
\end{align*}
%%%%%%%%%%%%%%%%%%%%%%%%%%%%%%%%%%%%%%%%%%%%%%%%%%%%%%%%%%%%%%%%%%%%%%%%%%%%%%%%%%%%%%%%%%%

In polar coordinates, at any $Z\in \mathbb{H}_{g}$, the Siegel metric is given by the following formula
\begin{align}\label{polar-metric}
\mu_{\mathrm{S}}^{\mathrm{Vol}_{\mathrm{S}}}(Z):=\prod_{j=1}^{n}\sinh^{2}(r_j)\prod_{1\leq j<k\leq g}\sinh^{2}((r_{j}-r_{k})/2)\prod_{1\leq j<k\leq g}\sinh^{2}((r_{j}+r_{k})/2)
\bigwedge_{j=1}^{g}\dd r_{j}\wedge \mu(q),
\end{align}
%%%%%%%%%%%%%%%%%%%%%%%%%%%%%%%%%%%%%%%%%%%%%%%%%%%%%%%%%%%%%%%%%%%%%%%%%%%%%%%%%%%%%%%%%%%
where $\mu(q)$ denotes the Haar measure on $K$.
%%%%%%%%%%%%%%%%%%%%%%%%%%%%%%%%%%%%%%%%%%%%%%%%%%%%%%%%%%%%%%%%%%%%%%%%%%%%%%%%%%%%%%%%%%%

Let $\vec{a}:=(a_1,\ldots,a_{g})$, where for each $1\leq j\leq g$, $a_{j}\in\mathbb{R}_{>0}$. Set
\begin{align}\label{dg}
\mathbb{D}_{g}(Z,\vec{a}):=\big\lbrace  W\in\mathbb{H}_{g}\big|\,|r_{j}(Z,W)|=\big|\tanh^{-1}\big(\sqrt{\rho_{j}(Z,W)}\big)\big|\leq a_{j},\,\,1\leq j\leq g \big\rbrace.
\end{align}
%%%%%%%%%%%%%%%%%%%%%%%%%%%%%%%%%%%%%%%%%%%%%%%%%%%%%%%%%%%%%%%%%%%%%%%%%%%%%%%%%%%%%%%%%%%
Let $\mathrm{Vol}_{\mathrm{S}}(\mathbb{D}_{g}(Z,\vec{a}))$ denote the volume of $\mathbb{D}_{g}(Z,\vec{a})$ with respect to the volume associated to the Siegel metric $\mu_{\mathrm{S}}$.  From equation \eqref{polar-metric}, we infer that
\begin{align}\label{polar-metric-2}
&\mathrm{Vol}_{\mathrm{S}}(\mathbb{D}_{g}(Z,\vec{a}))=\int_{\mathbb{D}_{g}(\vec{a})}\mu_{\mathrm{S}}^{\mathrm{Vol}_{\mathrm{S}}}(Z)=\notag \\&\int_{-a_{1}}^{a_{1}}\cdots\int_{-a_{g}}^{a_{g}} \prod_{j=1}^{n}\sinh^{2}(r_j)\prod_{1\leq j<k\leq g}\sinh^{2}((r_{j}-r_{k})/2)\prod_{1\leq j<k\leq g}\sinh^{2}((r_{j}+r_{k})/2)\bigwedge_{j=1}^{g}\dd r_{j}.
\end{align}
%%%%%%%%%%%%%%%%%%%%%%%%%%%%%%%%%%%%%%%%%%%%%%%%%%%%%%%%%%%%%%%%%%%%%%%%%%%%%%%%%%%%%%%%%%%
%%%%%%%%%%%%%%%%%%%%%%%%%%%%%%%%%%%%%%%%%%%%%%%%%%%%%%%%%%%%%%%%%%%%%%%%%%%%%%%%%%%%%%%%%%%

\vspace{0.2cm}
\section{Estimates}\label{sec-3}
%%%%%%%%%%%%%%%%%%%%%%%%%%%%%%%%%%%%%%%%%%%%%%%%%%%%%%%%%%%%%%%%%%%%%%%%%%%%%%%%%%%%%%%%%%%
In this section, we prove a  Jorgenson-Lundelius type estimate from \cite{jl}, using which we prove Theorem \ref{mainthm1} and Theorem \ref{mainthm2}.
%%%%%%%%%%%%%%%%%%%%%%%%%%%%%%%%%%%%%%%%%%%%%%%%%%%%%%%%%%%%%%%%%%%%%%%%%%%%%%%%%%%%%%%%%%%
%%%%%%%%%%%%%%%%%%%%%%%%%%%%%%%%%%%%%%%%%%%%%%%%%%%%%%%%%%%%%%%%%%%%%%%%%%%%%%%%%%%%%%%%%%%
\subsection{Jorgenson-Lundelius type estimate}\label{sec-3.1}
%%%%%%%%%%%%%%%%%%%%%%%%%%%%%%%%%%%%%%%%%%%%%%%%%%%%%%%%%%%%%%%%%%%%%%%%%%%%%%%%%%%%%%%%%%%
In this section, we first compute an upper bound for the volume of a geodesic ball with radius $r$ in Siegel upper half space of genus $g$, using which we prove a Jorgenson-Lundelius type 
estimate from \cite{jl}.
%%%%%%%%%%%%%%%%%%%%%%%%%%%%%%%%%%%%%%%%%%%%%%%%%%%%%%%%%%%%%%%%%%%%%%%%%%%%%%%%%%%%%%%%%%%

In the following lemma, we derive an upper bound for the volume of $D_{g}(\vec{r})$.
%%%%%%%%%%%%%%%%%%%%%%%%%%%%%%%%%%%%%%%%%%%%%%%%%%%%%%%%%%%%%%%%%%%%%%%%%%%%%%%%%%%%%%%%%%%

\vspace{0.1cm}
\begin{prop}\label{prop1}
With notation as above, for any $r\in\mathbb{R}_{>0}$ with $\vec{r}=(r,r)^{t}\in\mathbb{R}_{>0}^{2}$, we have the following estimate
\begin{align}\label{prop1-eqn}
\mathrm{Vol}_{\mathrm{S}}(\mathbb{D}_{2}(Z,r))\leq 32  \cosh^2{r}\sinh^4{r}.
\end{align}
%%%%%%%%%%%%%%%%%%%%%%%%%%%%%%%%%%%%%%%%%%%%%%%%%%%%%%%%%%%%%%%%%%%%%%%%%%%%%%%%%%%%%%%%%%%
\begin{proof}
From equation \eqref{polar-metric-2}, we have
\begin{align*}
 &\mathrm{Vol}_{\mathrm{S}}(\mathbb{D}_{2}(Z,\vec{r}))=\int_{\mathbb{D}_{2}(Z,\vec{r})}\mu_{\mathrm{S}}(Z)=\notag\\&4\int_{0}^{r}\int_{0}^{r}\sinh^{2}(r_1)\sinh^{2}(r_2)\sinh^{2}((r_{1}-r_{2})/2)\sinh^{2}((r_{1}+r_{2})/2)\dr1\dr2.
\end{align*}
%%%%%%%%%%%%%%%%%%%%%%%%%%%%%%%%%%%%%%%%%%%%%%%%%%%%%%%%%%%%%%%%%%%%%%%%%%%%%%%%%%%%%%%%%%%
Substituting  the formula $2\sin((x+y)/2)\sinh((x-y)/2)=\cosh(x)-\cosh(y)$ on the right hand-side of the above inequality, we now compute
\begin{align}\label{prop1-eqn1}
\mathrm{Vol}_{\mathrm{S}}(\mathbb{D}_{2}(Z,r))\leq  \mathcal{I}_{1}+\mathcal{I}_{2}-2\mathcal{I}_3,
\end{align}
%%%%%%%%%%%%%%%%%%%%%%%%%%%%%%%%%%%%%%%%%%%%%%%%%%%%%%%%%%%%%%%%%%%%%%%%%%%%%%%%%%%%%%%%%%%
where 
\begin{align}\label{prop1-eqn2}
&\mathcal{I}_{1}:=\int_{0}^{r}\int_{0}^{r} \sinh^2{r_1}\cosh^2{r_1}\sinh^2{r_2}\dr1\dr2;\notag\\ &\mathcal{I}_{2}:=\int_{0}^{r}\int_{0}^{r} \sinh^2{r_1}\sinh^2{r_2}\cosh^2{r_2}\dr1\dr2;\notag\\
&\mathcal{I}_{3}:=\int_{0}^{r}\int_{0}^{r}\sinh^2{r_1}\sinh^2{r_2}\cosh{r_1}\cosh{r_2} \dr1\dr2.
\end{align}
%%%%%%%%%%%%%%%%%%%%%%%%%%%%%%%%%%%%%%%%%%%%%%%%%%%%%%%%%%%%%%%%%%%%%%%%%%%%%%%%%%%%%%%%%%%
We now evaluate each of the above integrals. For the first integral, we have
\begin{align}\label{prop1-eqn3}
\mathcal{I}_1 = \int_{0}^{r} \int_{0}^{r} \sinh^2{r_1}\cosh^2{r_1}\sinh^2{r_2} \dr1 \dr2 
=\int_{0}^{r}  \sinh^2{r_2}\dr2 \int_{0}^{r} \sinh^2{r_1}\cosh^2{r_1}\dr1=\notag\\
\frac{1}{48}( \sinh{r}\cosh{r} - r )( 7\sinh^3{r}\cosh{r} + 3\sinh{r}\cosh{r}  - 3r ).
\end{align}
%%%%%%%%%%%%%%%%%%%%%%%%%%%%%%%%%%%%%%%%%%%%%%%%%%%%%%%%%%%%%%%%%%%%%%%%%%%%%%%%%%%%%%%%%%%
For the second integral, we have 
\begin{align}\label{prop1-eqn4}
\mathcal{I}_2 = \int_{0}^{r}\int_{0}^{r} \sinh^2{r_2}\cosh^2{r_2}\sinh^2{r_1} \dr1 \dr2 =
\int_{0}^{r} \sinh^2{r_1}\dr1 \int_{0}^{r} \sinh^2{r_2}\cosh^2{r_2}\dr2=\notag\\
\frac{1}{48}( \sinh{r}\cosh{r} - r )( 7\sinh^3{r}\cosh{r} + 3\sinh{r}\cosh{r}  - 3r ).
\end{align}
%%%%%%%%%%%%%%%%%%%%%%%%%%%%%%%%%%%%%%%%%%%%%%%%%%%%%%%%%%%%%%%%%%%%%%%%%%%%%%%%%%%%%%%%%%%
For the third integral, we have
\begin{align}\label{prop1-eqn5}
\mathcal{I}_3 = \int_{0}^{r}  \int_{0}^{r} \sinh^2{r_1}\sinh^2{r_2}\cosh{r_1}\cosh{r_2} \dr1 \dr2 
=\notag\\ \int_{0}^{r}   \sinh^2{r_1}\cosh{r_1}\dr1 \int_{0}^{r}   \sinh^2{r_2}\cosh{r_2}\dr2= \frac{1}{9}\sinh^6{r}.
\end{align}
%%%%%%%%%%%%%%%%%%%%%%%%%%%%%%%%%%%%%%%%%%%%%%%%%%%%%%%%%%%%%%%%%%%%%%%%%%%%%%%%%%%%%%%%%%%
Combining equations \eqref{prop1-eqn1}--\eqref{prop1-eqn5}, we arrive at the following estimate
\begin{align*}
\mathrm{Vol}_{\mathrm{S}}(\mathbb{D}_{2}(Z,\vec{r}))\leq 32\cosh^2{r}\sinh^4{r},
\end{align*}
%%%%%%%%%%%%%%%%%%%%%%%%%%%%%%%%%%%%%%%%%%%%%%%%%%%%%%%%%%%%%%%%%%%%%%%%%%%%%%%%%%%%%%%%%%%
which completes the proof of the proposition.
\end{proof}
\end{prop}
%%%%%%%%%%%%%%%%%%%%%%%%%%%%%%%%%%%%%%%%%%%%%%%%%%%%%%%%%%%%%%%%%%%%%%%%%%%%%%%%%%%%%%%%%%%
%%%%%%%%%%%%%%%%%%%%%%%%%%%%%%%%%%%%%%%%%%%%%%%%%%%%%%%%%%%%%%%%%%%%%%%%%%%%%%%%%%%%%%%%%%%

\begin{prop}\label{prop2}
With notation as above, for any $r\in\mathbb{R}_{>0}$ with $\vec{r}=(r,\ldots,r)^{t}\in\mathbb{R}_{>0}^{g}$, we have the following estimate
\begin{align}\label{prop2-eqn}
\mathrm{Vol}_{\mathrm{S}}(\mathbb{D}_{g}(Z,\vec{r}))=O_{g}\big( (\cosh{r})^{g^2-2}(\sinh{r})^{g+2} \big),
\end{align}
%%%%%%%%%%%%%%%%%%%%%%%%%%%%%%%%%%%%%%%%%%%%%%%%%%%%%%%%%%%%%%%%%%%%%%%%%%%%%%%%%%%%%%%%%%%
where the implied constant depends only on $g$.
%%%%%%%%%%%%%%%%%%%%%%%%%%%%%%%%%%%%%%%%%%%%%%%%%%%%%%%%%%%%%%%%%%%%%%%%%%%%%%%%%%%%%%%%%%%
\begin{proof}
From equation \eqref{polar-metric-2}, we have
\begin{align*}
&\mathrm{Vol}_{\mathrm{S}}(\mathbb{D}_{g}(Z,\vec{r}))=\int_{\mathbb{D}_{g}(Z,\vec{r})}\mu_{\mathrm{S}}(Z)=\notag\\&2^{g}\int_{0}^{r}\cdots\int_{0}^{r}
\prod_{j=1}^{g}\sinh^{2}r_j\prod_{1\leq j<k\leq g}\sinh^{2}((r_{j}-r_{k})/2)\prod_{1\leq j<k\leq g}\sinh^{2}((r_{j}+r_{k})/2)\bigwedge_{j=1}^{g}\dd r_{j}.
 \end{align*}
%%%%%%%%%%%%%%%%%%%%%%%%%%%%%%%%%%%%%%%%%%%%%%%%%%%%%%%%%%%%%%%%%%%%%%%%%%%%%%%%%%%%%%%%%%%
For $1\leq j,k\leq g$, we have the estimate
\begin{align*}
\sinh^2((r_j + r_k)/2)\sinh^2((r_j - r_k)/2) & = \frac{1}{4}( \cosh{r_j} - \cosh{r_k})^2   \leq \frac{1}{4}( \cosh^2{r_j} + \cosh^2{r_k} ).
\end{align*}
%%%%%%%%%%%%%%%%%%%%%%%%%%%%%%%%%%%%%%%%%%%%%%%%%%%%%%%%%%%%%%%%%%%%%%%%%%%%%%%%%%%%%%%%%%%
So, we have
\begin{align*}
\mathrm{Vol}_{\mathrm{S}}(\mathbb{D}_{g}(z,\vec{r}))\leq  2^{g}\int_{0}^{r}\cdots\int_{0}^{r}\prod_{j = 1}^g \sinh^2{r_j}\prod_{1\leq j<k \leq g} ( \cosh^2{r_j} + \cosh^2{r_k} ) \bigwedge_{j = 1}^g \dd r_j.
\end{align*}
%%%%%%%%%%%%%%%%%%%%%%%%%%%%%%%%%%%%%%%%%%%%%%%%%%%%%%%%%%%%%%%%%%%%%%%%%%%%%%%%%%%%%%%%%%%
Observe that
\begin{align}\label{prop2-eqn1}
\mathrm{Vol}_{\mathrm{S}}(\mathbb{D}_{g}(z,\vec{r}))\leq\int_{\mathbb{D}_{g}(Z,\vec{r})}\mu_{\mathrm{S}}(Z)\leq 2^g\int_{0}^{r}\cdots   \int_{0}^{r}\prod_{j = 1}^g \sinh^2{(r_i)}\prod_{1\leq j<k \leq n} ( \cosh^2{r_j} + \cosh^2{r_k} ) =\notag \\\int_{0}^{r}\cdots   \int_{0}^{r}\bigg(\int_{0}^{r}\sinh^2{r_g}\prod_{j=l}^{g-1}(\cosh^2{r_g}+\cosh^2{r_j})\dd {r_g}\bigg) \times
\notag\\\bigg(\prod_{j = 1}^{g-1} \sinh^2{r_j}\prod_{1\leq j<k \leq g-1} ( \cosh^2{r_j} + \cosh^2{r_k} )\bigg)\bigwedge_{j = 1}^{g-1} \dd r_j.
\end{align}
%%%%%%%%%%%%%%%%%%%%%%%%%%%%%%%%%%%%%%%%%%%%%%%%%%%%%%%%%%%%%%%%%%%%%%%%%%%%%%%%%%%%%%%%%%%

We now proceed by induction. For any $y, x_1,\ldots,x_{m},\in\mathbb{R}_{>0}$ with $x=(x_1,\ldots x_m)$, we have
\begin{align*}
\prod_{j=1}^m (x_i + y) = \sum_{k=0}^m \bigg( \sum_{1\leq j_1 < \dots < j_{m-k} \leq m} x_{j_1} \dots x_{j_{m-k}} \bigg)y^k= \sum_{k=0}^n e_{m,m-k}(x) y^k\\[0.1cm]
\mathrm{where}\,\,e_{m,k}(x) = \sum_{1\leq j_1 < \dots < j_k \leq m} x_{j_1} \dots x_{j_k}.
\end{align*}
%%%%%%%%%%%%%%%%%%%%%%%%%%%%%%%%%%%%%%%%%%%%%%%%%%%%%%%%%%%%%%%%%%%%%%%%%%%%%%%%%%%%%%%%%%%
Set
\begin{align*}
 P_{g}(r_1,\ldots,r_{g}):=\prod_{j = 1}^g \sinh^2{r_j}\prod_{1\leq j<k \leq g} ( \cosh^2{r_j} + \cosh^2{r_k} ).
\end{align*}
%%%%%%%%%%%%%%%%%%%%%%%%%%%%%%%%%%%%%%%%%%%%%%%%%%%%%%%%%%%%%%%%%%%%%%%%%%%%%%%%%%%%%%%%%%%
From estimate \eqref{prop2-eqn1}, we have
\begin{align}\label{prop2-eqn2}
&\int_{\mathbb{D}_{g}(Z,\vec{r})}\mu_{\mathrm{S}}(Z)\leq 
 \int_{\mathbb{D}_{g}(Z,\vec{r})} P_g(r_1,\ldots,r_{g}) \bigwedge_{j = 1}^g \dd r_j =\notag\\[0.09cm]&
 \int_{\mathbb{D}_{g-1}(Z,\vec{r})} \bigg(\int_0^r \sinh^2{r_g}\prod_{j=1}^{g-1}(\cosh^2{r_g}+\cosh^2{r_j}) \dd r_g\bigg)P_{g-1}(r_{1},\ldots,r_{g-1})\bigwedge_{j = 1}^{g-1} \dd r_j .
\end{align}
%%%%%%%%%%%%%%%%%%%%%%%%%%%%%%%%%%%%%%%%%%%%%%%%%%%%%%%%%%%%%%%%%%%%%%%%%%%%%%%%%%%%%%%%%%%
Put $x=(\cosh^2{r_1},\ldots,\cosh^2{r_{g-1}})$. We now compute
\begin{align*}
\int_0^r \sinh^2{r_g}\prod_{j=l}^{g-1}(\cosh^2{r_g}+\cosh^2{r_j}) \dd r_g %
=\int_0^r \sinh^2{r_g}\sum_{k = 0}^{g-1}e_{g-1,g-1-k} (x)\cosh^{2k}{r_g} \dd r_g= \notag\\
\sum_{k = 0}^{g-1}e_{g-1,g-1-k}(x) \int_0^r \sinh^2{r_g}\cosh^{2k}{r_g} \dd r_g.
\end{align*}
%%%%%%%%%%%%%%%%%%%%%%%%%%%%%%%%%%%%%%%%%%%%%%%%%%%%%%%%%%%%%%%%%%%%%%%%%%%%%%%%%%%%%%%%%%%
From application of formula 2.4.12  \cite{grad}, we deduce the following estimate
\begin{align*}
2 \int_0^{r} \sinh^2{r_g}\cosh^{2k}{r_g} \dd r_{g} \leq \sinh^3{r} \cosh^{2k-1}{r} + \sinh{r}\cosh{r},
\end{align*}
%%%%%%%%%%%%%%%%%%%%%%%%%%%%%%%%%%%%%%%%%%%%%%%%%%%%%%%%%%%%%%%%%%%%%%%%%%%%%%%%%%%%%%%%%%%
using which, we arrive at the following estimate
\begin{align}\label{prop3-eqn3}
\int_0^r \sinh^2{r_g}\prod_{j=1}^{g-1}(\cosh^2{r_g}+\cosh^2{r_j}) \dd r_g \leq \notag\\ \sum_{k = 0}^{g-1}e_{g-1,g-1-k}(x) \big( \sinh^3{r} \cosh^{2k-1}{r} + \sinh{r}\cosh{r} \big).
\end{align}
%%%%%%%%%%%%%%%%%%%%%%%%%%%%%%%%%%%%%%%%%%%%%%%%%%%%%%%%%%%%%%%%%%%%%%%%%%%%%%%%%%%%%%%%%%%
For $(r_1,\ldots,r_{g-1})\in \mathcal{D}_{g-1}(Z,\vec{r})$, we have $0\leq r_1,\ldots,r_{g-1}\leq r$. From estimate \eqref{prop3-eqn3}, for $x=(\cosh^{2}{r_1},\ldots,\cosh^{2}{r_{g-1}})$,  we derive
\begin{align}\label{prop2-eqn4}
\sum_{k = 0}^{g-1} \binom{g-1}{g-1-k}(\cosh{r})^{2(g-1-k)} \big( \sinh^3{r} \cosh^{2k-1}{r} + \sinh{r}\cosh{r} \big) =\notag\\
\sinh^3{r}\cosh^{2g-3}{r}\bigg(\sum_{k = 0}^{g-1} \binom{g-1}{k}\bigg) + \sinh{r}\bigg( \sum_{k = 0}^{g-1} \binom{g-1}{k} \cosh^{2g-2k-1}{r}\bigg) =\notag\\
2^{g-1}\sinh^3{r}\cosh^{2g-3}{r} + 2^{g-1}\sinh{r}\cosh^{2g-1}{r} < C_{g} \sinh{r}\cosh^{2g-1}{r},
\end{align}
%%%%%%%%%%%%%%%%%%%%%%%%%%%%%%%%%%%%%%%%%%%%%%%%%%%%%%%%%%%%%%%%%%%%%%%%%%%%%%%%%%%%%%%%%%%
for some constant $C_{g}$, which only depends on $g$. 
%%%%%%%%%%%%%%%%%%%%%%%%%%%%%%%%%%%%%%%%%%%%%%%%%%%%%%%%%%%%%%%%%%%%%%%%%%%%%%%%%%%%%%%%%%%

Hence, combining estimates \eqref{prop2-eqn1}--\eqref{prop2-eqn4}, arrive at the following estimate
\begin{align}\label{prop2-eqn5}
\mathrm{Vol}_{\mathrm{S}}(\mathbb{D}_{g}(Z,\vec{r}))\leq\int_{\mathbb{D}_{g}(Z,\vec{r})}\mu_{\mathrm{S}}(Z)
\leq \int_{\mathbb{D}_{g}(Z,\vec{r})} P_g(r_1,\ldots r_{g}) \bigwedge_{j = 1}^g \dd r_j \leq\notag\\[0.1cm] C_g \sinh{r}\cosh^{2g-1}{r}\int_{\mathbb{D}_{g-1}(Z,\vec{r})} P_{g-1}(r_1,\ldots r_{g-1}) \bigwedge_{j = 1}^{g-1} \dd r_j.
\end{align}
%%%%%%%%%%%%%%%%%%%%%%%%%%%%%%%%%%%%%%%%%%%%%%%%%%%%%%%%%%%%%%%%%%%%%%%%%%%%%%%%%%%%%%%%%%%
We now assume that
\begin{align}\label{prop2-eqn6}
\int_{\mathbb{D}_{g-1}(Z,\vec{r})}\mu_{\mathrm{S}}(Z) \leq C_{g-1}(\cosh{r})^{(g-1)^2-2}(\sinh{r})^{g+1}.
\end{align}
%%%%%%%%%%%%%%%%%%%%%%%%%%%%%%%%%%%%%%%%%%%%%%%%%%%%%%%%%%%%%%%%%%%%%%%%%%%%%%%%%%%%%%%%%%%

From estimate \eqref{prop1-eqn}, clearly induction hypothesis is satisfied for $g=2$. Combining estimates \eqref{prop2-eqn5} and \eqref{prop2-eqn6}, we arrive at the following estimate
\begin{align}\label{prop2-eqn7}
\mathrm{Vol}_{\mathrm{S}}(\mathbb{D}_{g}(Z,\vec{r}))=\int_{\mathbb{D}_{g}(Z,\vec{r})}\mu_{\mathrm{S}}(Z)\leq C_g (\cosh{r})^{g^2-2}(\sinh{r})^{g+2},
\end{align}
%%%%%%%%%%%%%%%%%%%%%%%%%%%%%%%%%%%%%%%%%%%%%%%%%%%%%%%%%%%%%%%%%%%%%%%%%%%%%%%%%%%%%%%%%%%
which completes the proof of the proposition. 
\end{proof}
\end{prop}
%%%%%%%%%%%%%%%%%%%%%%%%%%%%%%%%%%%%%%%%%%%%%%%%%%%%%%%%%%%%%%%%%%%%%%%%%%%%%%%%%%%%%%%%%%%

We now prove a Jorgenson-Lundelius type estimate, which is an extension of Lemma 3 and 4 to the setting of Siegel modular varieties from \cite{jl}. 
%%%%%%%%%%%%%%%%%%%%%%%%%%%%%%%%%%%%%%%%%%%%%%%%%%%%%%%%%%%%%%%%%%%%%%%%%%%%%%%%%%%%%%%%%%%

For any $Z,W\in\mathbb{H}_{g}$, when $\Gamma$ is cocompact, we define
\begin{align*}
\mathcal{N}_\Gamma(Z, W; \rho) := \# \big\lbrace\gamma \in \Gamma\backslash \lbrace\mathrm{Id}_{g}\rbrace  \big|\  d_{\mathrm{S}}(Z, \gamma W) < \rho \big\rbrace;
\end{align*}
%%%%%%%%%%%%%%%%%%%%%%%%%%%%%%%%%%%%%%%%%%%%%%%%%%%%%%%%%%%%%%%%%%%%%%%%%%%%%%%%%%%%%%%%%%%
when $\Gamma$ be is an arithmetic subgroup, we define
\begin{align*}
\mathcal{N}_\Gamma(Z, W; \rho) := \# \big\lbrace\gamma \in \Gamma\backslash \Gamma_{\infty}  \big|\  d_{\mathrm{S}}(Z, \gamma W) < \rho \big\rbrace.
\end{align*}
%%%%%%%%%%%%%%%%%%%%%%%%%%%%%%%%%%%%%%%%%%%%%%%%%%%%%%%%%%%%%%%%%%%%%%%%%%%%%%%%%%%%%%%%%%%

Furthermore, for brevity of notation, here after, for any $r\in\mathbb{R}_{>0}$ we use the following notation
\begin{align}
\mathbb{D}_{g}(Z,r):=\mathbb{D}_{g}(Z,\vec{r}),\,\,\mathrm{where}\,\,\vec{r}:=(r,\ldots, r)^{t}\in\mathbb{R}_{>}^{g}.
\end{align}
%%%%%%%%%%%%%%%%%%%%%%%%%%%%%%%%%%%%%%%%%%%%%%%%%%%%%%%%%%%%%%%%%%%%%%%%%%%%%%%%%%%%%%%%%%%
The following lemma is an extension of Lemma 3 from \cite{jl} to the setting of Siegel modular varieties. 
%%%%%%%%%%%%%%%%%%%%%%%%%%%%%%%%%%%%%%%%%%%%%%%%%%%%%%%%%%%%%%%%%%%%%%%%%%%%%%%%%%%%%%%%%%%

\vspace{0.2cm}
\begin{lem}\label{lem3}
With the notation as above, let $\Gamma\subset\mathrm{Sp}(2g,\mathbb{R})$ be a cocompact subgroup or an arithmetic subgroup with $r_{\Gamma}$ is as defined in equations  \eqref{irad-1} or \eqref{irad-2}, respectively. Then, for  $\rho_{0}\in\mathbb{R}_{>0}$ with $\rho_{0}>r_{\Gamma}$, we have
\begin{align}\label{lem3-eqn}
\mathcal{N}_\Gamma(Z,W; \rho) \leq \mathcal{N}_\Gamma(Z, W; \rho_{0}) +
\frac{\mathrm{Vol}_{\mathrm{S}}(\mathbb{D}_{g}(Z,(\rho + r_{\Gamma})/2\sqrt{2}))-\mathrm{Vol}_{\mathrm{S}}(\mathbb{B}_{g}(Z,\rho_{0}-r_{\Gamma}))}{\mathrm{Vol}_{\mathrm{S}}(\mathbb{B}_{g}(Z,r_{\Gamma}))}.
\end{align}
%%%%%%%%%%%%%%%%%%%%%%%%%%%%%%%%%%%%%%%%%%%%%%%%%%%%%%%%%%%%%%%%%%%%%%%%%%%%%%%%%%%%%%%%%%%
\end{lem}
%%%%%%%%%%%%%%%%%%%%%%%%%%%%%%%%%%%%%%%%%%%%%%%%%%%%%%%%%%%%%%%%%%%%%%%%%%%%%%%%%%%%%%%%%%%
\begin{proof}
From the definition of $\mathcal{N}_\Gamma(Z, W; \rho)$, it follows that
\begin{align*}
\mathcal{N}_\Gamma(Z,W; \rho) \leq \mathcal{N}_\Gamma(Z, W; \rho_{0}) +
\frac{\mathrm{Vol}_{\mathrm{S}}(\mathbb{B}_{g}(Z,\rho + r_{\Gamma}))-\mathrm{Vol}_{\mathrm{S}}(\mathbb{B}_{g}(Z, \rho_{0}-r_{\Gamma}))}{\mathrm{Vol}_{\mathrm{S}}(\mathbb{B}_{g}(Z, r_{\Gamma}))}.
\end{align*}
%%%%%%%%%%%%%%%%%%%%%%%%%%%%%%%%%%%%%%%%%%%%%%%%%%%%%%%%%%%%%%%%%%%%%%%%%%%%%%%%%%%%%%%%%%%

The proof of the lemma follows from the observation 
\begin{align*}
\mathrm{Vol}_{\mathrm{S}}(\mathbb{B}_{g}(Z, \rho + r_{\Gamma}))\subset\mathrm{Vol}_{\mathrm{S}}(\mathbb{D}_{g}(Z, (\rho + r_{\Gamma})/2\sqrt{2})).
\end{align*}
\end{proof}
%%%%%%%%%%%%%%%%%%%%%%%%%%%%%%%%%%%%%%%%%%%%%%%%%%%%%%%%%%%%%%%%%%%%%%%%%%%%%%%%%%%%%%%%%%%

\begin{thm}\label{thm4}
Let $f$ be a smooth, positive, and  monotonically decreasing function on $\RR_{>0}$. Then, for $\Gamma$ cocompact with $r_{\Gamma}$ as defined in equation \eqref{irad-1}, and for 
$\rho_{0}\in\mathbb{R}_{>0}$ with $\rho_{0}>r_{\Gamma}$, we have the following estimate
\begin{align}\label{thm4-eqn-1}
&\sum_{\gamma\in \Gamma\backslash\lbrace\mathrm{Id}\rbrace}f(d_{\mathrm{S}}(Z,\gamma W))
\leq  \int_0^{\rho_{0}} f(\rho) \dd N_\Gamma(Z, W; \rho) +f(\rho_{0}) \frac{\mathrm{Vol}_{\mathrm{S}}(\mathbb{D}_{g}(Z,\rho_{0}/2\sqrt{2}))}{\mathrm{Vol}_{\mathrm{S}}(\mathbb{B}_{g}(Z,r_{\Gamma}))} +\notag\\ & \frac{C_{g}}{\mathrm{Vol}_{\mathrm{S}}(\mathbb{B}_{g}(Z,r_{\Gamma}))} \int_{\rho_{0}}^\infty f(\rho)(\cosh((\rho+r_{\Gamma})/2\sqrt{2}))^{g^2-1}(\sinh((\rho+r_{\Gamma})/2\sqrt2))^{g+1} \dd\rho.
\end{align}
%%%%%%%%%%%%%%%%%%%%%%%%%%%%%%%%%%%%%%%%%%%%%%%%%%%%%%%%%%%%%%%%%%%%%%%%%%%%%%%%%%%%%%%%%%%
Then, for $\Gamma$ an arithmetic subgroup with $r_{\Gamma}$ as defined in equation \eqref{irad-2}, and for $\rho_{0}\in\mathbb{R}_{>0}$ with $\rho_{0}>r_{\Gamma}$, we have the following estimate
\begin{align}\label{thm4-eqn-2}
&\sum_{\gamma\in \Gamma\backslash\Gamma_{\infty}}f(d_{\mathrm{S}}(Z,\gamma W)) \leq 
 \int_0^{\rho_{0}} f(\rho) \dd N_\Gamma(Z, W; \rho) +f(\rho_{0}) \frac{\mathrm{Vol}_{\mathrm{S}}(\mathbb{D}_{g}(Z,\rho_{0}/2\sqrt{2}))}{\mathrm{Vol}_{\mathrm{S}}(\mathbb{B}_{g}(Z,r_{\Gamma}))} + \notag\\ &\frac{C_{g}}{\mathrm{Vol}_{\mathrm{S}}(\mathbb{B}_{g}(Z,r_{\Gamma}))} \int_{\rho_{0}}^\infty f(\rho)(\cosh((\rho+r_{\Gamma})/2\sqrt{2}))^{g^2-1}(\sinh((\rho+r_{\Gamma})/2\sqrt{2}))^{g+1} \dd\rho,
\end{align}
%%%%%%%%%%%%%%%%%%%%%%%%%%%%%%%%%%%%%%%%%%%%%%%%%%%%%%%%%%%%%%%%%%%%%%%%%%%%%%%%%%%%%%%%%%%
and $C_{g}$ in both the above estimates is a constant, which only depends on $g$.
\end{thm}
%%%%%%%%%%%%%%%%%%%%%%%%%%%%%%%%%%%%%%%%%%%%%%%%%%%%%%%%%%%%%%%%%%%%%%%%%%%%%%%%%%%%%%%%%%%
\begin{proof}
We prove the theorem for the case when $\Gamma$ is cocompact, i.e., the proof for estimate \eqref{thm4-eqn-1}. The proof for estimate \eqref{thm4-eqn-2} is analogous. 
%%%%%%%%%%%%%%%%%%%%%%%%%%%%%%%%%%%%%%%%%%%%%%%%%%%%%%%%%%%%%%%%%%%%%%%%%%%%%%%%%%%%%%%%%%%

For $\rho_{0}>r_{\Gamma}$, from Lemma \ref{lem3}, we have
\begin{align}\label{thm4-eqn1}
\sum_{\gamma\in \Gamma\backslash\Gamma_{\infty}}f(d_{\mathrm{S}}(Z,\gamma W))=\int_{0}^{\infty}f(\rho)\dd N_\Gamma(Z, W; \rho)\leq 
 \int_0^{\rho_{0}} f(\rho) \dd N_\Gamma(Z, W; \rho) +\notag\\ f(\rho_{0}) \frac{\mathrm{Vol}_{\mathrm{S}}(\mathbb{D}_{g}(Z,\rho_{0}/2\sqrt{2}))}{\mathrm{Vol}_{\mathrm{S}}(\mathbb{B}_{g}(Z,r_{\Gamma}))} + \frac{1}{\mathrm{Vol}_{\mathrm{S}}(\mathbb{B}_{g}(Z,r_{\Gamma}))} \int_{\rho_{0}}^\infty f(\rho)\dd \mathrm{Vol}_{\mathrm{S}}(\mathbb{D}_{g}(Z,r_{\Gamma})).
\end{align}
%%%%%%%%%%%%%%%%%%%%%%%%%%%%%%%%%%%%%%%%%%%%%%%%%%%%%%%%%%%%%%%%%%%%%%%%%%%%%%%%%%%%%%%%%%%
For any $r\in\mathbb{R}_{>0}$, from estimate \eqref{prop2-eqn6}, from the proof of Proposition \ref{prop2}, we have 
\begin{align*}
\dd\mathrm{Vol}_{\mathrm{S}}(\mathbb{D}_{g}(Z,r))=C_{g} \dd ((\cosh{r})^{g^2-2}(\sinh{r})^{g+2}) =C_{g}(g^2-2)(\cosh{r})^{g^2-3}(\sinh{r})^{g+3}\dd r +
 \\C_{g} (g+2)(\cosh{r})^{g^2-1}(\sinh{r})^{g+1}\dd r\leq C_g (\cosh{r})^{g^2-1}(\sinh{r})^{g+1} \dd r.
\end{align*}
%%%%%%%%%%%%%%%%%%%%%%%%%%%%%%%%%%%%%%%%%%%%%%%%%%%%%%%%%%%%%%%%%%%%%%%%%%%%%%%%%%%%%%%%%%

So, we can conclude that
\begin{align}\label{thm4-eqn2}
\dd \mathrm{Vol}_{\mathrm{S}}(\mathbb{D}_{g}(Z,(\rho+ r_{\Gamma})/2\sqrt{2}))\leq C_g (\cosh((\rho+r_{\Gamma})/2\sqrt{2}))^{g^2-1}(\sinh((\rho+r_{\Gamma})/2\sqrt{2})^{g+1} \dd r.
\end{align}
%%%%%%%%%%%%%%%%%%%%%%%%%%%%%%%%%%%%%%%%%%%%%%%%%%%%%%%%%%%%%%%%%%%%%%%%%%%%%%%%%%%%%%%%%%%
Combining estimates \eqref{thm4-eqn1} and \eqref{thm4-eqn2}, completes the proof of the Theorem.
\end{proof}
%%%%%%%%%%%%%%%%%%%%%%%%%%%%%%%%%%%%%%%%%%%%%%%%%%%%%%%%%%%%%%%%%%%%%%%%%%%%%%%%%%%%%%%%%%%
%%%%%%%%%%%%%%%%%%%%%%%%%%%%%%%%%%%%%%%%%%%%%%%%%%%%%%%%%%%%%%%%%%%%%%%%%%%%%%%%%%%%%%%%%%%

\vspace{0.2cm}
\subsection{Proofs of Theorem \ref{mainthm1} and Theorem \ref{mainthm2}}\label{sec-3.2}
%%%%%%%%%%%%%%%%%%%%%%%%%%%%%%%%%%%%%%%%%%%%%%%%%%%%%%%%%%%%%%%%%%%%%%%%%%%%%%%%%%%%%%%%%%%
Using all the notation and results, which are set up in the previous sections, we now prove Theorem \ref{mainthm1} and Theorem \ref{mainthm2}.  
%%%%%%%%%%%%%%%%%%%%%%%%%%%%%%%%%%%%%%%%%%%%%%%%%%%%%%%%%%%%%%%%%%%%%%%%%%%%%%%%%%%%%%%%%%%

For the rest of the section, we use the following notation. Let $f,\,h:\mathbb{H}_{g}\longrightarrow \mathbb{R}$ be two real valued functions. We denote 
\begin{align}
f\ll_{g} h,
\end{align}
 if there exists a constant $C_{g}$, which only depends on genus $g$ of $\mathbb{H}_{g}$, such that $f(Z)\leq C_{g} h(Z)$, for all $Z\in\mathbb{H}_{g}$, and $C_{g}$ is called the implied constant 
%%%%%%%%%%%%%%%%%%%%%%%%%%%%%%%%%%%%%%%%%%%%%%%%%%%%%%%%%%%%%%%%%%%%%%%%%%%%%%%%%%%%%%%%%%%

\vspace{0.2cm}
\begin{proof}[{\bf{Proof of Theorem \ref{mainthm1}}}]
%%%%%%%%%%%%%%%%%%%%%%%%%%%%%%%%%%%%%%%%%%%%%%%%%%%%%%%%%%%%%%%%%%%%%%%%%%%%%%%%%%%%%%%%%%%
With hypothesis as above, from estimate \eqref{bk-est2}, we have 
\begin{align}\label{thm1-eqn1}
\|\mathcal{B}_{X_{\Gamma}}^{\ell^{k}}(Z,W)\|_{\ell^k}\leq \sum_{\gamma\in\Gamma}\prod_{j=1}^{g}\frac{C_{g(k+1)}}{\cosh^{k(g+1)}(r_j(Z,\gamma W))}.
\end{align}
%%%%%%%%%%%%%%%%%%%%%%%%%%%%%%%%%%%%%%%%%%%%%%%%%%%%%%%%%%%%%%%%%%%%%%%%%%%%%%%%%%%%%%%%%%%
We now estimate the infinite series on the right hand-side of the above estimate. Given a $Z,W\in X_{\Gamma}$, we work with a fixed Dirichlet fundamental domain $\mathcal{F}_{\Gamma,Z}$ centered at $Z$, which is as described in equation \eqref{dir}. 
%%%%%%%%%%%%%%%%%%%%%%%%%%%%%%%%%%%%%%%%%%%%%%%%%%%%%%%%%%%%%%%%%%%%%%%%%%%%%%%%%%%%%%%%%%%

For $(x_1,\ldots x_g)^{t}\in\mathbb{R}_{>0}^{g}$ with $\sum_{j = 1}^g x_j^2 = x^2$, we have the elementary inequality
\begin{align*}
\cosh{x} \leq \prod_{j=1}^{g}\cosh{x_j} \leq \cosh^g(x/\sqrt{g}).
\end{align*}
%%%%%%%%%%%%%%%%%%%%%%%%%%%%%%%%%%%%%%%%%%%%%%%%%%%%%%%%%%%%%%%%%%%%%%%%%%%%%%%%%%%%%%%%%%%
Using which, for any $Z,W\in \mathcal{F}_ {\Gamma,Z}$, and $\gamma\in\Gamma$, we have the following estimate
\begin{align*}
\prod\limits_{j = 1}^g\frac{1}{\cosh^{k(g+1)}(r_j(Z, \gamma W))} 
\leq \frac{1}{\cosh^{k(g+1)}\big(\sum\limits_{j=1}^{k}r^2_j(Z, \gamma W)\big)^{1/2}}
\leq \frac{1}{\cosh^{k(g+1)}(d_{\mathrm{S}}(Z, \gamma W)/2 \sqrt{2})},
\end{align*}
%%%%%%%%%%%%%%%%%%%%%%%%%%%%%%%%%%%%%%%%%%%%%%%%%%%%%%%%%%%%%%%%%%%%%%%%%%%%%%%%%%%%%%%%%%%
which leads us to the following estimate 
\begin{align}\label{thm1-eqn2}
\sum_{\gamma\in\Gamma}\prod\limits_{j = 1}^g\frac{1}{\cosh^{k(g+1)}(r_j(Z, \gamma W))}  \leq \sum_{\gamma \in \Gamma}\frac{1}{\cosh^{k(g+1)}(d_{\mathrm{S}}(Z, \gamma W)/2 \sqrt{2})}.
\end{align}
%%%%%%%%%%%%%%%%%%%%%%%%%%%%%%%%%%%%%%%%%%%%%%%%%%%%%%%%%%%%%%%%%%%%%%%%%%%%%%%%%%%%%%%%%%%

Recall that the injectivity radius $r_{\Gamma}$ is as defined in equation \eqref{irad-1}. We now divide the proof into two cases. 
%%%%%%%%%%%%%%%%%%%%%%%%%%%%%%%%%%%%%%%%%%%%%%%%%%%%%%%%%%%%%%%%%%%%%%%%%%%%%%%%%%%%%%%%%%%

{\bf{Case 1: $\rho_{0}=d_{\mathrm{S}}(Z,W)>3r_{\Gamma}/2$.}}
%%%%%%%%%%%%%%%%%%%%%%%%%%%%%%%%%%%%%%%%%%%%%%%%%%%%%%%%%%%%%%%%%%%%%%%%%%%%%%%%%%%%%%%%%%%

From estimate \eqref{thm4-eqn-1}, we have
\begin{align}\label{thm1-eqn3}
&\sum_{\gamma \in \Gamma}\frac{1}{\cosh^{k(g+1)}(d_{\mathrm{S}}(Z, \gamma W)/2 \sqrt{2})}\ll_{g}\int_{0}^{\infty}\frac{\dd\mathcal{N}_{\Gamma}(Z,W;\rho)}{\cosh^{k(g+1)}(\rho/2 \sqrt{2})}\ll_{g} \notag\\[0.1cm]&\int_0^{\rho_{0}} \frac{\dd N_\Gamma(Z, W; \rho) }{\cosh^{k(g+1)}(\rho/2 \sqrt{2})}+ \frac{\mathrm{Vol}_{\mathrm{S}}(\mathbb{D}_{g}(Z,\rho_{0}/2\sqrt{2}))}{\mathrm{Vol}_{\mathrm{S}}(\mathbb{B}_{g}(Z,r_{\Gamma}))\cdot \cosh^{k(g+1)}(\rho_{0}/2 \sqrt{2})} +\notag\\[0.1cm]& \frac{1}{\mathrm{Vol}_{\mathrm{S}}(\mathbb{B}_{g}(Z,r_{\Gamma}))} \int_{\rho_{0}}^\infty\frac{(\cosh((\rho+r_{\Gamma})/2\sqrt{2}))^{g^2-1}(\sinh((\rho+r_{\Gamma})/2\sqrt{2}))^{g+1} \dd\rho}{\cosh^{k(g+1)}(\rho/2 \sqrt{2})},
\end{align}
%%%%%%%%%%%%%%%%%%%%%%%%%%%%%%%%%%%%%%%%%%%%%%%%%%%%%%%%%%%%%%%%%%%%%%%%%%%%%%%%%%%%%%%%%%%
where the implied constant depends only on $g$.
%%%%%%%%%%%%%%%%%%%%%%%%%%%%%%%%%%%%%%%%%%%%%%%%%%%%%%%%%%%%%%%%%%%%%%%%%%%%%%%%%%%%%%%%%%%

We now estimate each of the three terms on the right hand-side of the inequality in estimate \eqref{thm1-eqn3}. 
%%%%%%%%%%%%%%%%%%%%%%%%%%%%%%%%%%%%%%%%%%%%%%%%%%%%%%%%%%%%%%%%%%%%%%%%%%%%%%%%%%%%%%%%%%%

For the first term, from the choice of the fundamental domain $\mathcal{F}_{\Gamma}$, we have the following estimate
\begin{align}\label{thm1-eqn4}
\int_0^{\rho_{0}} \frac{\dd N_\Gamma(Z, W; \rho) }{\cosh^{k(g+1)}(\rho/2 \sqrt{2})}=\frac{1}{\cosh^{k(g+1)}(\rho_{0}/2 \sqrt{2})}=\frac{1}{\cosh^{k(g+1)}(d_{\mathrm{S}}(Z,W)/2 \sqrt{2})}.
\end{align} 
%%%%%%%%%%%%%%%%%%%%%%%%%%%%%%%%%%%%%%%%%%%%%%%%%%%%%%%%%%%%%%%%%%%%%%%%%%%%%%%%%%%%%%%%%%%
For the second term, using estimate \eqref{prop2-eqn}, we derive
\begin{align}\label{thm1-eqn5}
\frac{\mathrm{Vol}_{\mathrm{S}}(\mathbb{D}_{g}(Z,\rho_{0}/2\sqrt{2}))}{\mathrm{Vol}_{\mathrm{S}}(\mathbb{B}_{g}(Z,r_{\Gamma}))\cdot \cosh^{k(g+1)}(\rho_{0}/2 \sqrt{2})}\ll_{g} \frac{(\cosh (\rho_{0}/2\sqrt{2}))^{g^2-2}(\sinh (\rho_0/2\sqrt{2}))^{g+2}}{\mathrm{Vol}_{\mathrm{S}}(\mathbb{B}_{g}(Z,r_{\Gamma}))\cdot \cosh^{k(g+1)}(\rho_{0}/2\sqrt{2})}\ll_{g} \notag\\[0.1cm] 
\frac{1}{\mathrm{Vol}_{\mathrm{S}}(\mathbb{B}_{g}(Z,r_{\Gamma}))\cdot \cosh^{k(g+1)-g^{2}-g}(d_{\mathrm{S}}(Z,W)/2\sqrt{2})},
\end{align}
%%%%%%%%%%%%%%%%%%%%%%%%%%%%%%%%%%%%%%%%%%%%%%%%%%%%%%%%%%%%%%%%%%%%%%%%%%%%%%%%%%%%%%%%%%%
where the implied constant depends only on $g$.
%%%%%%%%%%%%%%%%%%%%%%%%%%%%%%%%%%%%%%%%%%%%%%%%%%%%%%%%%%%%%%%%%%%%%%%%%%%%%%%%%%%%%%%%%%%

For the third term, we compute
\begin{align}\label{thm1-eqn6}
&\frac{1}{\mathrm{Vol}_{\mathrm{S}}(\mathbb{B}_{g}(Z,r_{\Gamma}))} \int_{\rho_{0}}^\infty\frac{(\cosh((\rho+r_{\Gamma})/2\sqrt{2}))^{g^2-1}(\sinh((\rho+r_{\Gamma})/2\sqrt{2}))^{g+1} \dd\rho}{\cosh^{k(g+1)}(\rho/2 \sqrt{2})}\ll_{g}\notag\\[0.1cm]&\frac{1}{\mathrm{Vol}_{\mathrm{S}}(\mathbb{B}_{g}(Z,r_{\Gamma}))} \int_{\rho_{0}}^\infty\frac{\sinh((\rho+r_{\Gamma})/2\sqrt{2}) \dd\rho}{\cosh^{k(g+1)-g^{2}-g+1}(\rho/2 \sqrt{2})}=\notag\\[0.1cm]&\frac{1}{(k(g+1)-g^{2}-g+2)\mathrm{Vol}_{\mathrm{S}}(\mathbb{B}_{g}(Z,r_{\Gamma}))}\cdot\frac{1}{\cosh^{k(g+1)-g^{2}-g+2}(d_{\mathrm{S}}(Z,W)/2\sqrt{2})},
\end{align}
%%%%%%%%%%%%%%%%%%%%%%%%%%%%%%%%%%%%%%%%%%%%%%%%%%%%%%%%%%%%%%%%%%%%%%%%%%%%%%%%%%%%%%%%%%%
where the implied constant depends only on $g$.
%%%%%%%%%%%%%%%%%%%%%%%%%%%%%%%%%%%%%%%%%%%%%%%%%%%%%%%%%%%%%%%%%%%%%%%%%%%%%%%%%%%%%%%%%%%

Combining estimates \eqref{thm1-eqn1}--\eqref{thm1-eqn6}, we arrive at the following estimate
\begin{align*}
\|\mathcal{B}_{X_{\Gamma}}^{\ell^{k}}(Z,W)\|_{\ell^k}=O_{X_{\Gamma}}\bigg( \frac{C_{g(k+1)}}{\cosh^{k(g+1)-g^{2}-g}(d_{\mathrm{S}}(Z,W)/2\sqrt{2})}\bigg).
\end{align*}
%%%%%%%%%%%%%%%%%%%%%%%%%%%%%%%%%%%%%%%%%%%%%%%%%%%%%%%%%%%%%%%%%%%%%%%%%%%%%%%%%%%%%%%%%%%
where the implied constant depends only on $X_{\Gamma}$, which completes the proof of estimate \eqref{mainthm1-eqn}, for the case $d_{\mathrm{S}}(Z,W)>3r_{\Gamma}/2$.
%%%%%%%%%%%%%%%%%%%%%%%%%%%%%%%%%%%%%%%%%%%%%%%%%%%%%%%%%%%%%%%%%%%%%%%%%%%%%%%%%%%%%%%%%%%

{\bf{Case 2: $\rho_{0}=d_{\mathrm{S}}(Z,W)<3r_{\Gamma}/2$.}}
%%%%%%%%%%%%%%%%%%%%%%%%%%%%%%%%%%%%%%%%%%%%%%%%%%%%%%%%%%%%%%%%%%%%%%%%%%%%%%%%%%%%%%%%%%%

From estimate \eqref{thm4-eqn-1}, we have
\begin{align}\label{thm1-eqn7}
&\sum_{\gamma \in \Gamma}\frac{1}{\cosh^{k(g+1)}(d_{\mathrm{S}}(Z, \gamma W)/2 \sqrt{2})}\ll_{g} \int_{0}^{\infty}\frac{\dd\mathcal{N}_{\Gamma}(Z,W;\rho)}{\cosh^{k(g+1)}(\rho/2 \sqrt{2})}\ll_{g}\int_0^{\rho_{0}} \frac{\dd N_\Gamma(Z, W; \rho) }{\cosh^{k(g+1)}(\rho/2 \sqrt{2})}+  \notag\\[0.1cm]&\int_{\rho_{0}}^{3r_{\Gamma/2}}\frac{\dd N_\Gamma(Z, W; \rho) }{\cosh^{k(g+1)}(\rho/2 \sqrt{2})}+\frac{\mathrm{Vol}_{\mathrm{S}}(\mathbb{D}_{g}(Z,3r_{\Gamma}/4\sqrt{2}))}{\mathrm{Vol}_{\mathrm{S}}(\mathbb{B}_{g}(Z,r_{\Gamma}))\cdot \cosh^{k(g+1)}(3r_{\Gamma}/4 \sqrt{2})} +\notag\\[0.1cm]& \frac{1}{\mathrm{Vol}_{\mathrm{S}}(\mathbb{B}_{g}(Z,r_{\Gamma}))} \int_{3r_{\Gamma}/2}^\infty\frac{(\cosh((\rho+r_{\Gamma})/2\sqrt{2}))^{g^2-1}(\sinh((\rho+r_{\Gamma})/2\sqrt{2}))^{g+1} \dd\rho}{\cosh^{k(g+1)}(\rho/2 \sqrt{2})},
\end{align}
%%%%%%%%%%%%%%%%%%%%%%%%%%%%%%%%%%%%%%%%%%%%%%%%%%%%%%%%%%%%%%%%%%%%%%%%%%%%%%%%%%%%%%%%%%%

where the implied constant depends only on $g$.
%%%%%%%%%%%%%%%%%%%%%%%%%%%%%%%%%%%%%%%%%%%%%%%%%%%%%%%%%%%%%%%%%%%%%%%%%%%%%%%%%%%%%%%%%%%

We now estimate each of the four terms on the right hand-side of the inequality in estimate \eqref{thm1-eqn7}. 
%%%%%%%%%%%%%%%%%%%%%%%%%%%%%%%%%%%%%%%%%%%%%%%%%%%%%%%%%%%%%%%%%%%%%%%%%%%%%%%%%%%%%%%%%%%

For the first term, from the choice of the fundamental domain $\mathcal{F}_{\Gamma}$, we have the following estimate
\begin{align}\label{thm1-eqn8}
\int_0^{\rho_{0}} \frac{\dd N_\Gamma(Z, W; \rho) }{\cosh^{k(g+1)}(\rho/2 \sqrt{2})}=\frac{1}{\cosh^{k(g+1)}(\rho_{0}/2 \sqrt{2})}=\frac{1}{\cosh^{k(g+1)}(d_{\mathrm{S}}(Z,W)/2 \sqrt{2})}.
\end{align} 
%%%%%%%%%%%%%%%%%%%%%%%%%%%%%%%%%%%%%%%%%%%%%%%%%%%%%%%%%%%%%%%%%%%%%%%%%%%%%%%%%%%%%%%%%%%
For the second term, using estimate \eqref{prop2-eqn}, we derive
\begin{align}\label{thm1-eqn9}
\int_{\rho_{0}}^{3r_{\Gamma/2}}\frac{\dd N_\Gamma(Z, W; \rho) }{\cosh^{k(g+1)}(\rho/2 \sqrt{2})}\leq\frac{\mathrm{Vol}_{\mathrm{S}}(\mathbb{D}_{g}(Z,3r_{\Gamma}/2))}{\mathrm{Vol}_{\mathrm{S}}(\mathbb{B}_{g}(Z,r_{\Gamma}))}\cdot\frac{1}{\cosh^{k(g+1)}(d_{\mathrm{S}}(Z,W)/2\sqrt{2})}\ll_{g}\notag\\[0.1cm]\frac{(\cosh (3r_{\Gamma}/4\sqrt{2}))^{g^2-2}(\sinh (3r_{\Gamma}/4\sqrt{2}))^{g+2}}{\mathrm{Vol}_{\mathrm{S}}(\mathbb{B}_{g}(Z,r_{\Gamma}))\cdot\cosh^{k(g+1)}(d_{\mathrm{S}}(Z,W)/2\sqrt{2})},
\end{align}
%%%%%%%%%%%%%%%%%%%%%%%%%%%%%%%%%%%%%%%%%%%%%%%%%%%%%%%%%%%%%%%%%%%%%%%%%%%%%%%%%%%%%%%%%%%
where the implied constant depends only on $g$.
%%%%%%%%%%%%%%%%%%%%%%%%%%%%%%%%%%%%%%%%%%%%%%%%%%%%%%%%%%%%%%%%%%%%%%%%%%%%%%%%%%%%%%%%%%%

For the third term, using estimate \eqref{prop2-eqn}, we derive
\begin{align}\label{thm1-eqn10}
\frac{\mathrm{Vol}_{\mathrm{S}}(\mathbb{D}_{g}(Z,3r_{\Gamma}/4\sqrt{2}))}{\mathrm{Vol}_{\mathrm{S}}(\mathbb{B}_{g}(Z,r_{\Gamma}))\cdot \cosh^{k(g+1)}(3r_{\Gamma}/4\sqrt{2})}\leq \frac{(\cosh (3r_{\Gamma}/4\sqrt{2}))^{g^2-2}(\sinh (3r_{\Gamma}/4\sqrt{2}))^{g+2}}{\mathrm{Vol}_{\mathrm{S}}(\mathbb{B}_{g}(Z,r_{\Gamma}))\cdot \cosh^{k(g+1)}(3r_{\Gamma}/4\sqrt{2})}\leq \notag\\[0.1cm] 
\frac{1}{\mathrm{Vol}_{\mathrm{S}}(\mathbb{B}_{g}(Z,r_{\Gamma}))\cdot \cosh^{k(g+1)-g^{2}-g}(3r_{\Gamma}/4\sqrt{2})},
\end{align}
%%%%%%%%%%%%%%%%%%%%%%%%%%%%%%%%%%%%%%%%%%%%%%%%%%%%%%%%%%%%%%%%%%%%%%%%%%%%%%%%%%%%%%%%%%%
where the implied constant depends only on $g$.
%%%%%%%%%%%%%%%%%%%%%%%%%%%%%%%%%%%%%%%%%%%%%%%%%%%%%%%%%%%%%%%%%%%%%%%%%%%%%%%%%%%%%%%%%%%

For the fourth term, we compute
\begin{align}\label{thm1-eqn11}
 &\frac{1}{\mathrm{Vol}_{\mathrm{S}}(\mathbb{B}_{g}(Z,r_{\Gamma}))} \int_{3r_{\Gamma/2}}^\infty\frac{(\cosh((\rho+r_{\Gamma})/2\sqrt{2}))^{g^2-1}(\sinh((\rho+r_{\Gamma})/2\sqrt{2}))^{g+1} \dd\rho}{\cosh^{k(g+1)}(\rho/2 \sqrt{2})}\leq\notag\\[0.1cm]&\frac{1}{\mathrm{Vol}_{\mathrm{S}}(\mathbb{B}_{g}(Z,r_{\Gamma}))} \int_{3r_{\Gamma}/2}^\infty\frac{\sinh((\rho+r_{\Gamma})/2\sqrt{2}) \dd\rho}{\cosh^{k(g+1)-g^{2}-g+1}(\rho/2 \sqrt{2})}=\notag\\[0.1cm]&\frac{1}{(k(g+1)-g^{2}-g+2)\mathrm{Vol}_{\mathrm{S}}(\mathbb{B}_{g}(Z,r_{\Gamma}))}\cdot\frac{1}{ \cosh^{k(g+1)-g^{2}-g+2}(3r_{\Gamma}/4\sqrt{2})}. 
\end{align}
%%%%%%%%%%%%%%%%%%%%%%%%%%%%%%%%%%%%%%%%%%%%%%%%%%%%%%%%%%%%%%%%%%%%%%%%%%%%%%%%%%%%%%%%%%%

Combining estimates \eqref{thm1-eqn1} and \eqref{thm1-eqn2} with estimates \eqref{thm1-eqn7}--\eqref{thm1-eqn11}, and estimate \eqref{bkseries}, we arrive at the following estimate
\begin{align*}
\|\mathcal{B}_{X_{\Gamma}}^{\ell^{k}}(Z,W)\|_{\ell^k}=O_{g}\bigg( \frac{C_{k(g+1)}}{\cosh^{k(g+1)-g^{2}-g}(d_{\mathrm{S}}(Z,W)/2\sqrt{2})}\bigg)=\\O_{X_{\Gamma}}\bigg(\frac{k^{g(g+1)/2}}{\cosh^{k(g+1)-g^{2}-g}(d_{\mathrm{S}}(Z,W)/2\sqrt{2})}\bigg),
\end{align*}
%%%%%%%%%%%%%%%%%%%%%%%%%%%%%%%%%%%%%%%%%%%%%%%%%%%%%%%%%%%%%%%%%%%%%%%%%%%%%%%%%%%%%%%%%%%
where the implied constant depends only on $X_{\Gamma}$, which completes the proof of estimate \eqref{mainthm1-eqn}, for the case $d_{\mathrm{S}}(Z,W)<3r_{\Gamma}/2$, and hence, completes the proof of Theorem \ref{mainthm1}.
\end{proof}
%%%%%%%%%%%%%%%%%%%%%%%%%%%%%%%%%%%%%%%%%%%%%%%%%%%%%%%%%%%%%%%%%%%%%%%%%%%%%%%%%%%%%%%%%%%
%%%%%%%%%%%%%%%%%%%%%%%%%%%%%%%%%%%%%%%%%%%%%%%%%%%%%%%%%%%%%%%%%%%%%%%%%%%%%%%%%%%%%%%%%%%

\vspace{0.2cm}
We now prove Theorem \ref{mainthm2}.
%%%%%%%%%%%%%%%%%%%%%%%%%%%%%%%%%%%%%%%%%%%%%%%%%%%%%%%%%%%%%%%%%%%%%%%%%%%%%%%%%%%%%%%%%%%

\begin{proof}[{\bf{Proof of Theorem \ref{mainthm2}}}]
%%%%%%%%%%%%%%%%%%%%%%%%%%%%%%%%%%%%%%%%%%%%%%%%%%%%%%%%%%%%%%%%%%%%%%%%%%%%%%%%%%%%%%%%%%%
As in proof of Theorem \ref{mainthm1}, given a $Z=X+iY,\,W=U+iV\in X_{\Gamma}$ (we identify $X_{\Gamma}$ with its universal cover $\mathbb{H}_{g}$), we now work with a fixed Dirichlet fundamental domain $\mathcal{F}_{\Gamma,Z}$ centered at $Z$. 
%%%%%%%%%%%%%%%%%%%%%%%%%%%%%%%%%%%%%%%%%%%%%%%%%%%%%%%%%%%%%%%%%%%%%%%%%%%%%%%%%%%%%%%%%%%

Combining estimates \eqref{bk-est2} and \eqref{thm1-eqn1}, we have
\begin{align}\label{thm2-eqn1}
\|\mathcal{B}_{X_{\Gamma}}^{\ell^{k}}(Z,W)\|_{\ell^k}\leq \sum_{\gamma\in\Gamma}\prod_{j=1}^{g}\frac{C_{g(k+1)}}{\cosh^{k(g+1)}(r_j(Z,\gamma W))}=\notag\\\sum_{\gamma \in \Gamma\backslash\Gamma_{\infty}\cup\lbrace \mathrm{Id}_{2g}\rbrace}\prod_{j=1}^{g}\frac{C_{g(k+1)}}{\cosh^{k(g+1)}(r_j(Z,\gamma W))}+\sum_{\gamma \in \Gamma_{\infty}\backslash\lbrace\mathrm{Id}_{2g}\rbrace}\prod_{j=1}^{g}\frac{C_{g(k+1)}}{\cosh^{k(g+1)}(r_j(Z,\gamma W))}.
\end{align}
%%%%%%%%%%%%%%%%%%%%%%%%%%%%%%%%%%%%%%%%%%%%%%%%%%%%%%%%%%%%%%%%%%%%%%%%%%%%%%%%%%%%%%%%%%%
We now estimate each of the two terms on the right hand-side of the inequality in estimate \eqref{thm2-eqn1}. 
%%%%%%%%%%%%%%%%%%%%%%%%%%%%%%%%%%%%%%%%%%%%%%%%%%%%%%%%%%%%%%%%%%%%%%%%%%%%%%%%%%%%%%%%%%%

For $k\gg 1$, for the first term, adapting arguments from  the proof of Theorem \ref{mainthm1}, we have the following estimate
\begin{align}\label{thm2-eqn2}
\sum_{\gamma \in \Gamma\backslash \Gamma_{\infty}\cup\lbrace \mathrm{Id}_{2g}\rbrace}\prod_{j=1}^{g}\frac{C_{g(k+1)}}{\cosh^{k(g+1)}(r_j(Z,\gamma W))}=O_{X_{\Gamma}}\bigg(\frac{k^{g(g+1)/2}}{\cosh^{k(g+1)-g^{2}-g}(d_{\mathrm{S}}(Z,W)/2\sqrt{2})}\bigg),
\end{align}
%%%%%%%%%%%%%%%%%%%%%%%%%%%%%%%%%%%%%%%%%%%%%%%%%%%%%%%%%%%%%%%%%%%%%%%%%%%%%%%%%%%%%%%%%%%
where the implied constant depends only on $X_{\Gamma}$.
%%%%%%%%%%%%%%%%%%%%%%%%%%%%%%%%%%%%%%%%%%%%%%%%%%%%%%%%%%%%%%%%%%%%%%%%%%%%%%%%%%%%%%%%%%%

Although we work with a Dirichlet fundamental domain centered at $Z$, through scaling matrices, without loss of generality, we may assume that  we are working with a Siegel  fundamental domain $\mathcal{F}_{\Gamma}^{\mathrm{Siegel}}$, which is as defined in section \ref{sec-2.2}. Furthermore, without loss of generality, we assume that $\mathrm{det}(Y)>\mathrm{det}(V)$.
%%%%%%%%%%%%%%%%%%%%%%%%%%%%%%%%%%%%%%%%%%%%%%%%%%%%%%%%%%%%%%%%%%%%%%%%%%%%%%%%%%%%%%%%%%%

We now estimate the second term. From equations \eqref{gamma-infty0} and \eqref{gamma-inftyj}, we have
\begin{align}\label{thm2-eqn3}
\sum_{\gamma \in \Gamma_{\infty}\backslash \lbrace\mathrm{Id}_{2g}\rbrace}\prod_{j=1}^{g}\frac{1}{\cosh^{k(g+1)}(r_j(Z,\gamma W))}=\sum_{j=0}^{g-1}\sum_{\gamma \in \Gamma_{\infty}^{j}\backslash \lbrace\mathrm{Id}_{2g}\rbrace}^{j}\prod_{j=1}^{g}\frac{1}{\cosh^{k(g+1)}(r_j(Z,\gamma W))},
\end{align}
%%%%%%%%%%%%%%%%%%%%%%%%%%%%%%%%%%%%%%%%%%%%%%%%%%%%%%%%%%%%%%%%%%%%%%%%%%%%%%%%%%%%%%%%%%%
where
\begin{align*}
&\Gamma_{\infty}^{0}:=\bigg\lbrace \begin{pmatrix} \mathrm{Id}_{g}&S\\ 0 &\mathrm{Id}_{g}\end{pmatrix}\bigg|\,S\in\mathrm{Sym}_{g}(\mathbb{Z})\bigg\rbrace,\\
&\mathrm{for} \,\,1\leq j\leq g-1,\,\,\,\Gamma_{\infty}^{j}:= \bigg\lbrace \begin{pmatrix} A&AS\\ 0 &A^{-t}\end{pmatrix}\bigg|\,A=\begin{pmatrix} \mathrm{Id}_{j}&0\\ L &\mathrm{Id}_{g-j}\end{pmatrix},\,\,S=\begin{pmatrix} 0&H^{t}\\ H &M\end{pmatrix}\bigg\rbrace,\notag \\[0.1cm]
&\mathrm{where}\,\,L,\,H\in\mathbb{Z}^{(g-j)\times j},\,M\in\mathrm{Sym}_{g-j}(\mathbb{Z}).
\end{align*}
%%%%%%%%%%%%%%%%%%%%%%%%%%%%%%%%%%%%%%%%%%%%%%%%%%%%%%%%%%%%%%%%%%%%%%%%%%%%%%%%%%%%%%%%%%%

We first estimate the term, which involves summation over elements of $\Gamma_{\infty}^{0}$. For $Z = X + iY,\, W = U+iV\in\mathcal{F}_{\Gamma}^{\mathrm{Siegel}}$, and  $\gamma = \big(\begin{smallmatrix}
    I & S \\ 0 & I
\end{smallmatrix}\big)\in \Gamma^0_\infty $, we have
\begin{align}\label{thm2-eqn4}
\gamma W = W + S = U+S + i V\implies (\gamma W - \bar{Z}) = ((U+S-X)+ i(V+Y)).
\end{align}
%%%%%%%%%%%%%%%%%%%%%%%%%%%%%%%%%%%%%%%%%%%%%%%%%%%%%%%%%%%%%%%%%%%%%%%%%%%%%%%%%%%%%%%%%%%
So, we get 
\begin{align} \label{thm2-eqn5}
&\mathrm{det}{\big((U+S-X)+ i(V+Y)\big)}^2 = \notag\\&\mathrm{det}{\big(\sqrt{V+Y}^{-1}(U+S-X)\sqrt{V+Y}^{-1} + i\mathrm{Id}_{g}\big)}^2\mathrm{det}{(V+Y)}^2
\end{align}
%%%%%%%%%%%%%%%%%%%%%%%%%%%%%%%%%%%%%%%%%%%%%%%%%%%%%%%%%%%%%%%%%%%%%%%%%%%%%%%%%%%%%%%%%%%
Furthermore, recall that for any real $g\times g$ matrix $A$ we have,
\begin{align} \label{thm2-eqn6}
& |\mathrm{det}{(A+ i \mathrm{Id}_{g})}| = |\mathrm{det}{(A-i\mathrm{Id}_{g})}|\implies\notag\\ & |\mathrm{det}{(A+i\mathrm{Id}_{g})}|^2 = |\mathrm{det}{(A+ i \mathrm{Id}_{g})}\mathrm{det}{(A-i\mathrm{Id}_{g})}| = |\mathrm{det}{(A^2 + \mathrm{Id}_{g})}|.
\end{align}
%%%%%%%%%%%%%%%%%%%%%%%%%%%%%%%%%%%%%%%%%%%%%%%%%%%%%%%%%%%%%%%%%%%%%%%%%%%%%%%%%%%%%%%%%%%

From observations described in equations \eqref{thm2-eqn4}--\eqref{thm2-eqn6}, we compute
\begin{align*}
&\sum_{\gamma \in \Gamma_{\infty}^{0}\backslash \lbrace\mathrm{Id}_{2g}\rbrace}\prod_{j=1}^{g}\frac{1}{\cosh^{k(g+1)}(r_j(Z,\gamma W))}= \sum_{\gamma \in \Gamma_\infty^0\backslash \lbrace\mathrm{Id}_{2g}\rbrace}\frac{(\mathrm{det}(4YV))^{k(g+1)/2}}{|\mathrm{det}{(\gamma W - \bar{Z})}|^{k(g+1)}}= \notag\\[0.1cm]
 &\sum_{\gamma \in \Gamma_\infty^0}\frac{(\mathrm{det}(4YV))^{k(g+1)/2}}{\big|\mathrm{det}{\big(\sqrt{V+Y}^{-1}(U+S-X)\sqrt{V+Y}^{-1} + i\mathrm{Id}_{g}\big)}^2\mathrm{det}{(V+Y)}^2\big|^{k(g+1)/2}} =\notag\\[0.1cm]
&\sum_{\gamma \in \Gamma_\infty^0}\frac{(\mathrm{det}(4YV))^{k(g+1)/2}}{\big|\mathrm{det}{\big((\sqrt{V+Y}^{-1}(U+S-X)\sqrt{V+Y}^{-1})^2 +\mathrm{Id}_{g} \big)}\mathrm{det}{(V+Y)}^2\big|^{k(g+1)/2}} .
 \end{align*}
%%%%%%%%%%%%%%%%%%%%%%%%%%%%%%%%%%%%%%%%%%%%%%%%%%%%%%%%%%%%%%%%%%%%%%%%%%%%%%%%%%%%%%%%%%%
Setting
 \begin{align*}T:= \sqrt{V+Y}^{-1}(U+S-X)\sqrt{V+Y}^{-1},
 \end{align*}
 %%%%%%%%%%%%%%%%%%%%%%%%%%%%%%%%%%%%%%%%%%%%%%%%%%%%%%%%%%%%%%%%%%%%%%%%%%%%%%%%%%%%%%%%%%%
we observe that
\begin{align}\label{thm2-eqn7}
&\sum_{\gamma \in \Gamma_{\infty}^{0}\backslash \lbrace\mathrm{Id}_{2g}\rbrace}\frac{1}{\cosh^{k(g+1)}(d_{\mathrm{S}}(Z, \gamma W)/2 \sqrt{2})}=\notag\\[0.1cm]&\frac{(\mathrm{det}(4YV))^{k(g+1)/2}}{\big(\mathrm{det}(Y+V) \big)^{k(g+1)}} \sum_{\gamma \in \Gamma_\infty^0}\frac{1}{\left|\mathrm{det}{(T^2 +\mathrm{Id}_{g} )}\right|^{k(g+1)/2}} 
 \leq \notag\\[0.1cm]&\frac{(\mathrm{det}(4YV))^{k(g+1)/2}}{\big(\mathrm{det}(Y+V) \big)^{k(g+1)}} \int_{S=S^t}\frac{[\dd S]}{\big|\mathrm{det}{(T^2 +\mathrm{Id}_{g} )}\big|^{k(g+1)/2}}.
\end{align}
 %%%%%%%%%%%%%%%%%%%%%%%%%%%%%%%%%%%%%%%%%%%%%%%%%%%%%%%%%%%%%%%%%%%%%%%%%%%%%%%%%%%%%%%%%%%
For our choice of $T$ as above, we derive
\begin{align*}
C_g[\dd T] =  \mathrm{det}{(Y+V)}^{-(g+1)/2}[\dd S],
\end{align*}
%%%%%%%%%%%%%%%%%%%%%%%%%%%%%%%%%%%%%%%%%%%%%%%%%%%%%%%%%%%%%%%%%%%%%%%%%%%%%%%%%%%%%%%%%%%
where $C_{g}$ is a constant which only depends on $g$. Using which, we deduce that
\begin{align}\label{thm2-eqn8}
 \frac{(\mathrm{det}(4YV))^{k(g+1)/2}}{\big(\mathrm{det}(Y+V) \big)^{k(g+1)}} \int_{S=S^t}\frac{[\dd S]}{\left|\mathrm{det}{(T^2 +\mathrm{Id}_{g} )}\right|^{k(g+1)/2}}\ll_{g}\notag\\[0.1cm] \frac{(\mathrm{det}(4YV))^{k(g+1)/2}}{\big(\mathrm{det}(Y+V) \big)^{k(g+1)-(g+1)/2}}\int_{T=T^t}\frac{[\dd T]}{\left|\mathrm{det}{(T^2 +\mathrm{Id}_{g} )}\right|^{k(g+1)/2}}.
\end{align}
%%%%%%%%%%%%%%%%%%%%%%%%%%%%%%%%%%%%%%%%%%%%%%%%%%%%%%%%%%%%%%%%%%%%%%%%%%%%%%%%%%%%%%%%%%%
where the implied constant depends only on $g$.
%%%%%%%%%%%%%%%%%%%%%%%%%%%%%%%%%%%%%%%%%%%%%%%%%%%%%%%%%%%%%%%%%%%%%%%%%%%%%%%%%%%%%%%%%%%

For $k\gg 1$, using Hua's matrix beta integral from p. 33 in \cite{hua},  we arrive at the following estimate
\begin{align}\label{thm2-eqn9}
& \int_{T=T^t}\frac{[\dd T]}{\left|\mathrm{det}{(T^2 +I )}\right|^{k(g+1)/2}}=\notag\\[0.1cm]&\pi^{\frac{g(g+1)}{4}}\frac{\Gamma\big( k(g+1)/2 - \frac{g}{2} \big)}{\Gamma\big( k(g+1)/2  \big)} \prod\limits_{j = 1}^{g-1}\frac{\Gamma\big( k(g+1) - \frac{g+j}{2} \big)}{\Gamma\big( k(g+1) - j \big)}=O_{g}\big(k^{-g(g+1)/4}\big),
 \end{align}
 %%%%%%%%%%%%%%%%%%%%%%%%%%%%%%%%%%%%%%%%%%%%%%%%%%%%%%%%%%%%%%%%%%%%%%%%%%%%%%%%%%%%%%%%%%%
where the implied constant depends only on $g$.
%%%%%%%%%%%%%%%%%%%%%%%%%%%%%%%%%%%%%%%%%%%%%%%%%%%%%%%%%%%%%%%%%%%%%%%%%%%%%%%%%%%%%%%%%%%

Combining estimates \eqref{thm2-eqn7}--\eqref{thm2-eqn9}, we arrive at the following estimate
\begin{align}\label{thm2-eqn10}
\sum_{\gamma \in \Gamma_{\infty}^{0}\backslash \lbrace\mathrm{Id}_{2g}\rbrace}\prod_{j=1}^{g}\frac{1}{\cosh^{k(g+1)}(r_j(Z,\gamma W))}\ll_{g} \frac{k^{-g(g+1)/4}(\mathrm{det}(4YV))^{k(g+1)/2}}{\big(\mathrm{det}(Y+V) \big)^{k(g+1)-(g+1)/2}},
\end{align}
%%%%%%%%%%%%%%%%%%%%%%%%%%%%%%%%%%%%%%%%%%%%%%%%%%%%%%%%%%%%%%%%%%%%%%%%%%%%%%%%%%%%%%%%%%%
where the implied constant depends only on $g$.
%%%%%%%%%%%%%%%%%%%%%%%%%%%%%%%%%%%%%%%%%%%%%%%%%%%%%%%%%%%%%%%%%%%%%%%%%%%%%%%%%%%%%%%%%%%

For $k\gg 1$, combining estimate \eqref{thm2-eqn10} with the elementary inequality
\begin{align*}
 \frac{1}{(\mathrm{det}{(V+Y)})^2} = \frac{1}{(\mathrm{det}(V))^2(\mathrm{det}{(I + C)})^2} \leq \frac{1}{(\mathrm{det}{(V)})^2},
\end{align*}
%%%%%%%%%%%%%%%%%%%%%%%%%%%%%%%%%%%%%%%%%%%%%%%%%%%%%%%%%%%%%%%%%%%%%%%%%%%%%%%%%%%%%%%%%%%
where $C:=V^{-1/2}YV^{-1/2}$,  we arrive at the following estimate
\begin{align}\label{thm2-eqn11}
\sum_{\gamma \in \Gamma_{\infty}^{0}\backslash \lbrace\mathrm{Id}_{2g}\rbrace}\prod_{j=1}^{g}\frac{1}{\cosh^{k(g+1)}(r_j(Z,\gamma W))}\ll_{g} \frac{k^{-g(g+1)/4}(\mathrm{det}{(4Y)})^{k(g+1)/2}}{\big(\mathrm{det}(V) \big)^{(k(g+1)-g-1)/2}},
\end{align}
%%%%%%%%%%%%%%%%%%%%%%%%%%%%%%%%%%%%%%%%%%%%%%%%%%%%%%%%%%%%%%%%%%%%%%%%%%%%%%%%%%%%%%%%%%%
where the implied constant depends only on $g$.
%%%%%%%%%%%%%%%%%%%%%%%%%%%%%%%%%%%%%%%%%%%%%%%%%%%%%%%%%%%%%%%%%%%%%%%%%%%%%%%%%%%%%%%%%%%

We now estimate the remaining terms in the series described in equation \eqref{thm2-eqn3}. For any $1\leq j\leq g-1$, $Z=X+iY,\,W=U+iV\in\mathcal{F}_{\Gamma}^{\mathrm{Siegel}}$,  and $\gamma\in \Gamma_{\infty}^{j}$, whose structure is as described in equation \eqref{gamma-inftyj}, we have
\begin{align*}
\gamma W-\overline{Z}= (A(U+S)A^t - X) + i(AVA^t + Y).
\end{align*}
%%%%%%%%%%%%%%%%%%%%%%%%%%%%%%%%%%%%%%%%%%%%%%%%%%%%%%%%%%%%%%%%%%%%%%%%%%%%%%%%%%%%%%%%%%%
Since the matrix $AVA^t + Y > 0$ (i.e., the matrix is positive definite), we infer that 
\begin{align*}
\mathrm{det}(AVA^t + Y)\mathrm{det}\big( (AVA^t + Y)^{-1/2}(A(U+S)A^t - X)(AVA^t + Y)^{-1/2} + i \mathrm{Id}_g  \big)  = & \notag \\[0.1cm]
\mathrm{det}(AVA^t + Y)\mathrm{det}\big( ( AVA^t + Y)^{-1/2}(A(U+S)A^t - X)(AVA^t + Y)^{-1/2})^2 +  \mathrm{Id}_g  \big)^{1/2},
\end{align*}
%%%%%%%%%%%%%%%%%%%%%%%%%%%%%%%%%%%%%%%%%%%%%%%%%%%%%%%%%%%%%%%%%%%%%%%%%%%%%%%%%%%%%%%%%%%
using which we derive
\begin{align}\label{thm2-eqn12}
\sum_{\gamma \in \Gamma_{\infty}^{j}\backslash \lbrace\mathrm{Id}_{2g}\rbrace}\prod_{j=1}^{g}\frac{1}{\cosh^{k(g+1)}(r_j(Z,\gamma W))} \leq \sum_A \bigg(\frac{(\mathrm{det}(4YV))^{k(g+1)/2}}{(\mathrm{det}(AVA^t + Y))^{k(g+1)}}\times \notag \\
 \int_{S\in\mathrm{Sym}_g(\mathbb{R})}\frac{[\dd S]}{\big(\mathrm{det}\big( ((AVA^t + Y)^{-1/2}S(AVA^t + Y)^{-1/2})^2 +  \mathrm{Id}_g  \big)\big)^{k(g+1)/2}}\bigg).
\end{align}
%%%%%%%%%%%%%%%%%%%%%%%%%%%%%%%%%%%%%%%%%%%%%%%%%%%%%%%%%%%%%%%%%%%%%%%%%%%%%%%%%%%%%%%%%%%
 Setting 
\begin{align*}
T:=(AVA^t + Y)^{-1/2}S(AVA^t + Y)^{-1/2},
\end{align*} 
%%%%%%%%%%%%%%%%%%%%%%%%%%%%%%%%%%%%%%%%%%%%%%%%%%%%%%%%%%%%%%%%%%%%%%%%%%%%%%%%%%%%%%%%%%%
we have 
\begin{align*}
C_g[\dd T] = \big(\mathrm{det}(AVA^t + Y)\big)^{-(g+1)/2}[\dd S],
\end{align*}
%%%%%%%%%%%%%%%%%%%%%%%%%%%%%%%%%%%%%%%%%%%%%%%%%%%%%%%%%%%%%%%%%%%%%%%%%%%%%%%%%%%%%%%%%%%
where $C_g$ is a constant  which depends only on $g$. 
%%%%%%%%%%%%%%%%%%%%%%%%%%%%%%%%%%%%%%%%%%%%%%%%%%%%%%%%%%%%%%%%%%%%%%%%%%%%%%%%%%%%%%%%%%%

So we arrive at the estimate
 \begin{align}\label{thm2-eqn13}
&\sum_{\gamma \in \Gamma_{\infty}^{j}\backslash \lbrace\mathrm{Id}_{2g}\rbrace}\prod_{j=1}^{g}\frac{1}{\cosh^{k(g+1)}(r_j(Z,\gamma W))} \ll_{g}\notag \\[0.1cm] &\sum_A \frac{(\mathrm{det}(4YV))^{k(g+1)/2}}{\big(\mathrm{det}(AVA^t + Y)\big)^{k(g+1)}} \int_{T = T^t}\frac{\mathrm{det}(AVA^t + Y)^{(g+1)/2}[\dd T]}{\mathrm{det}(T^2 + \mathrm{Id}_g)^{k(g+1)/2}},
\end{align}
%%%%%%%%%%%%%%%%%%%%%%%%%%%%%%%%%%%%%%%%%%%%%%%%%%%%%%%%%%%%%%%%%%%%%%%%%%%%%%%%%%%%%%%%%%%
where the implied constant depends only on $g$.
%%%%%%%%%%%%%%%%%%%%%%%%%%%%%%%%%%%%%%%%%%%%%%%%%%%%%%%%%%%%%%%%%%%%%%%%%%%%%%%%%%%%%%%%%%%

For $k\gg 1$, as in estimate \eqref{thm2-eqn9}, using Hua's Beta integral (p. 33 from \cite{hua}), we derive
\begin{align}\label{thm2-eqn14}
\sum_A \frac{(\mathrm{det}(4YV))^{k(g+1)/2}}{\big(\mathrm{det}(AVA^t + Y)\big)^{k(g+1)}} \int_{T = T^t}\frac{\big(\mathrm{det}(AVA^t + Y)\big)^{(g+1)/2}[\dd T]}{\mathrm{det}\big( T^2 + \mathrm{Id}_g\big)^{k(g+1)/2}}\ll_{g}\notag\\[0.1cm]
\sum_A \frac{k^{-g(g+1)/4}(\mathrm{det}(4YV))^{k(g+1)/2}}{\big(\mathrm{det}(AVA^t + Y)\big)^{k(g+1)-(g+1)/2}},
\end{align}
%%%%%%%%%%%%%%%%%%%%%%%%%%%%%%%%%%%%%%%%%%%%%%%%%%%%%%%%%%%%%%%%%%%%%%%%%%%%%%%%%%%%%%%%%%%
where the implied constant depends only on $g$
%%%%%%%%%%%%%%%%%%%%%%%%%%%%%%%%%%%%%%%%%%%%%%%%%%%%%%%%%%%%%%%%%%%%%%%%%%%%%%%%%%%%%%%%%%%

Since $\mathrm{det}(A) = 1$, we  have
\begin{align*}
\mathrm{det}(AVA^t + Y) = \mathrm{det}(A^{-1}YA^{-t} + V) = \mathrm{det}(V)\mathrm{det}(V^{-1/2}A^{-1}YA^{-t}V^{-1/2} + \mathrm{Id}_g).
\end{align*}
%%%%%%%%%%%%%%%%%%%%%%%%%%%%%%%%%%%%%%%%%%%%%%%%%%%%%%%%%%%%%%%%%%%%%%%%%%%%%%%%%%%%%%%%%%%
Observe that  
\begin{align*}
&V^{-1/2}A^{-1}YA^{-t}V^{-1/2} > 0\\
&\big\lbrace \big( V^{-1/2}A^{-1}YA^{-t}V^{-1/2} \big)^{1/2}\big|\,A\,\mathrm{comes\,\,from\,\,elements\,\,of\,\,\Gamma_{\infty}^{j}}\big\rbrace\subset \mathrm{Sym_g(\mathbb{R})}.
\end{align*} 
%%%%%%%%%%%%%%%%%%%%%%%%%%%%%%%%%%%%%%%%%%%%%%%%%%%%%%%%%%%%%%%%%%%%%%%%%%%%%%%%%%%%%%%%%%%
For $k\gg 1$, as before, we set 
\begin{align*}
Q = \big( V^{-1/2}A^{-1}YA^{-t}V^{-1/2} \big)^{1/2},
\end{align*}
%%%%%%%%%%%%%%%%%%%%%%%%%%%%%%%%%%%%%%%%%%%%%%%%%%%%%%%%%%%%%%%%%%%%%%%%%%%%%%%%%%%%%%%%%%%
and using Hua's Beta integral (p. 33 from \cite{hua}), and from arguments as the ones used in estimate \eqref{thm2-eqn9}, we compute
\begin{align}\label{thm2-eqn15}
&\sum_A \frac{k^{-g(g+1)/4}(\mathrm{det}(4YV))^{k(g+1)/2}}{\big(\mathrm{det}(AVA^t + Y)\big)^{k(g+1)-(g+1)/2}}\ll_{g}\notag\\[0.1cm]& \frac{k^{-g(g+1)/4}(\mathrm{det}(4YV))^{k(g+1)/2}}{\big(\mathrm{det}(V)\big)^{k(g+1) - (g+1)/2}}
\int_{Q = Q^t}\frac{[\dd Q]}{\mathrm{det}( Q^2 + \mathrm{Id}_g)^{k(g+1)/2 - (g+1)/4}}\ll_{g}\notag\\[0.1cm]&\frac{k^{-g(g+1)/2}( \mathrm{det}(4Y) )^{k(g+1)/2}}{\big(\mathrm{det}(V)\big)^{(k(g+1) -g-1)/2}},
\end{align}
%%%%%%%%%%%%%%%%%%%%%%%%%%%%%%%%%%%%%%%%%%%%%%%%%%%%%%%%%%%%%%%%%%%%%%%%%%%%%%%%%%%%%%%%%%%
where the implied constant depends only on $g$.
%%%%%%%%%%%%%%%%%%%%%%%%%%%%%%%%%%%%%%%%%%%%%%%%%%%%%%%%%%%%%%%%%%%%%%%%%%%%%%%%%%%%%%%%%%%

For $k\gg 1$, combining estimates \eqref{thm2-eqn12}-\eqref{thm2-eqn15} with estimate \eqref{bkseries}, we arrive at the following estimate
\begin{align}\label{thm2-eqn16}
\sum_{j=1}^{g-1}\sum_{\gamma \in \Gamma_{\infty}^{j}\backslash \lbrace\mathrm{Id}_{2g}\rbrace}\prod_{j=1}^{g}\frac{C_{k(g+1)}}{\cosh^{k(g+1)}(r_j(Z,\gamma W))} \ll_{g} \frac{( \mathrm{det}(4Y) )^{k(g+1)/2}}{\big(\mathrm{det}(V)\big)^{(k(g+1) -g-1)/2}},
\end{align}
%%%%%%%%%%%%%%%%%%%%%%%%%%%%%%%%%%%%%%%%%%%%%%%%%%%%%%%%%%%%%%%%%%%%%%%%%%%%%%%%%%%%%%%%%%%
where the implied constant depends only on $g$.
%%%%%%%%%%%%%%%%%%%%%%%%%%%%%%%%%%%%%%%%%%%%%%%%%%%%%%%%%%%%%%%%%%%%%%%%%%%%%%%%%%%%%%%%%%%

Hence, for $k\gg 1$, combining estimates \eqref{thm2-eqn1},  \eqref{thm2-eqn2}, \eqref{thm2-eqn3}, \eqref{thm2-eqn11}  and \eqref{thm2-eqn16}, we arrive at the following estimate
\begin{align*}
\|\mathcal{B}_{X_{\Gamma}}^{\ell^{k}}(Z,W)\|_{\ell^k}=O_{X_{\Gamma}}\bigg(\frac{k^{g(g+1)/4}( \mathrm{det}{(4Y)})^{k(g+1)/2}}{\big(\mathrm{det}(V) \big)^{(k(g+1)-g-1)/2}}+\frac{k^{g(g+1)/2}}{\cosh^{k(g+1)-g^{2}-g}(d_{\mathrm{S}}(z,w)/2\sqrt{2})}\bigg),
\end{align*}
%%%%%%%%%%%%%%%%%%%%%%%%%%%%%%%%%%%%%%%%%%%%%%%%%%%%%%%%%%%%%%%%%%%%%%%%%%%%%%%%%%%%%%%%%%%
where the implied constant depends only on $X_{\Gamma}$, and completes the proof of the theorem.
\end{proof}
%%%%%%%%%%%%%%%%%%%%%%%%%%%%%%%%%%%%%%%%%%%%%%%%%%%%%%%%%%%%%%%%%%%%%%%%%%%%%%%%%%%%%%%%%%%
%%%%%%%%%%%%%%%%%%%%%%%%%%%%%%%%%%%%%%%%%%%%%%%%%%%%%%%%%%%%%%%%%%%%%%%%%%%%%%%%%%%%%%%%%%%

\vspace{0.2cm}
\subsection*{Acknowledgements} The first author expresses his gratitude to J. Kramer and Anna von Pippich, and both the authors thank A. Mandal, for helpful discussions on estimates
of Bergman kernels associated to arithmetic manifolds. 
%%%%%%%%%%%%%%%%%%%%%%%%%%%%%%%%%%%%%%%%%%%%%%%%%%%%%%%%%%%%%%%%%%%%%%%%%%%%%%%%%%%%%%%%%%%
%%%%%%%%%%%%%%%%%%%%%%%%%%%%%%%%%%%%%%%%%%%%%%%%%%%%%%%%%%%%%%%%%%%%%%%%%%%%%%%%%%%%%%%%%%%

%%%%%%%%%%%%%%%%%%%%%%%%%%%%%%%%%%%%%%%%%%%%%%%%%%%%%%%%%%%%%%%%%%%%%%%%%%%%%%%%%%%%%%%%%%%
%%%%%%%%%%%%%%%%%%%%%%%%%%%%%%%%%%%%%%%%%%%%%%%%%%%%%%%%%%%%%%%%%%%%%%%%%%%%%%%%%%%%%%%%%%%
%%%%%%%%%%%%%%%%%%%%%%%%%%%%%%%%%%%%%%%%%%%%%%%%%%%%%%%%%%%%%%%%%%%%%%%%%%%%%%%%%%%%%%%%%%%

\begin{thebibliography}{AMM16}
%%%%%%%%%%%%%%%%%%%%%%%%%%%%%%%%%%%%%%%%%%%%%%%%%%%%%%%%%%%%%%%%%%%%%%%%%%%%%%%%%%%%%%%%%%%
\bibitem[AH25]{anil4} A. Aryasomayajula and G. Harinarayanan,
\newblock{\it{Sub-convexity estimates of Siegel cusp forms of genus two}},
\newblock{In preparation}.
%%%%%%%%%%%%%%%%%%%%%%%%%%%%%%%%%%%%%%%%%%%%%%%%%%%%%%%%%%%%%%%%%%%%%%%%%%%%%%%%%%%%%%%%%%%

\bibitem[AM18]{anil1} A. Aryasomayajula and P. Majumder,
\newblock{\it{Off-diagonal estimates of the Bergman kernel on hyperbolic Riemann surfaces of finite volume}},
\newblock{Proc. Amer. Math. Soc. 146, 4009--4020, 2018}.
%%%%%%%%%%%%%%%%%%%%%%%%%%%%%%%%%%%%%%%%%%%%%%%%%%%%%%%%%%%%%%%%%%%%%%%%%%%%%%%%%%%%%%%%%%%

\bibitem[AM20]{anil2} A. Aryasomayajula and P. Majumder,
\newblock{\it{Estimates of the Bergman kernel on a hyperbolic Riemann surface of finite volume-II}},
\newblock{Ann. Fac. Sci. Toulouse Math. 29, 795--804, 2020}.
%%%%%%%%%%%%%%%%%%%%%%%%%%%%%%%%%%%%%%%%%%%%%%%%%%%%%%%%%%%%%%%%%%%%%%%%%%%%%%%%%%%%%%%%%%%

\bibitem[ARS]{anil3} A. Aryasomayajula, D. Roy, D. Sadhukhan,
\newblock{\it{Estimates of Bergman kernels and Bergman metric on compact Picard surfaces}},
\newblock{J. Math. Anal. Appl. 534, 2024.}
%%%%%%%%%%%%%%%%%%%%%%%%%%%%%%%%%%%%%%%%%%%%%%%%%%%%%%%%%%%%%%%%%%%%%%%%%%%%%%%%%%%%%%%%%%%

\bibitem[Cr13]{christ} M. Christ,
\newblock{\it{Upper bounds for Bergman kernels associated to positive line bundles with smooth hermitian metrics}},
\newblock{arXiv:1308.0062, 2013}.
%%%%%%%%%%%%%%%%%%%%%%%%%%%%%%%%%%%%%%%%%%%%%%%%%%%%%%%%%%%%%%%%%%%%%%%%%%%%%%%%%%%%%%%%%%%

\bibitem[DLM06]{dai} X. Dai, K. Liu, and X. Ma,
\newblock{\it{On the asymptotic expansion of Bergman kernel}},
\newblock{J. Differential Geom., 72, 1--41, 2006}.
%%%%%%%%%%%%%%%%%%%%%%%%%%%%%%%%%%%%%%%%%%%%%%%%%%%%%%%%%%%%%%%%%%%%%%%%%%%%%%%%%%%%%%%%%%%

\bibitem[De98]{delin} H. Delin,
\newblock{\it{Pointwise estimates for the weighted Bergman projection kernel in $\mathbb{C}^n$, using a weighted $L^2$-estimate for the $\partial$-equation}},
\newblock{Ann. Inst. Fourier (Grenoble) 48, 967--997, 1998}.
%%%%%%%%%%%%%%%%%%%%%%%%%%%%%%%%%%%%%%%%%%%%%%%%%%%%%%%%%%%%%%%%%%%%%%%%%%%%%%%%%%%%%%%%%%%

\bibitem[GR81]{grad} I. Gradshteyn and I. Ryzhik, 
\newblock{\it{Tables of Integrals, Series, and Products}},
\newblock{Academic Press, 1981}.
%%%%%%%%%%%%%%%%%%%%%%%%%%%%%%%%%%%%%%%%%%%%%%%%%%%%%%%%%%%%%%%%%%%%%%%%%%%%%%%%%%%%%%%%%%%

\bibitem[Hu63]{hua} Luogeng Hua, 
\newblock{\it{Harmonic analysis of functions of several complex variables in the classical domains.}},
\newblock{Translations of mathematical monographs, v. 6., Providence, American Mathematical Society, 1963.}
%%%%%%%%%%%%%%%%%%%%%%%%%%%%%%%%%%%%%%%%%%%%%%%%%%%%%%%%%%%%%%%%%%%%%%%%%%%%%%%%%%%%%%%%%%%

\bibitem[JL95]{jl} J. Jorgenson and R. Lundelius,
\newblock{\it{Convergence of the heat kernel and the resolvent kernel on degenerating hyperbolic Riemann surfaces of finite volume}},
\newblock{Quaestiones Mathematicae, 18, 345--363, 1995}.
 %%%%%%%%%%%%%%%%%%%%%%%%%%%%%%%%%%%%%%%%%%%%%%%%%%%%%%%%%%%%%%%%%%%%%%%%%%%%%%%%%%%%%%%%%%%

\bibitem[KM23]{mandal} J. Kramer and A. Mandal,
\newblock{\it{Uniform sup-norm bounds on average for Siegel cusp forms}},
\newblock{arXiv:2310.05334}.
 %%%%%%%%%%%%%%%%%%%%%%%%%%%%%%%%%%%%%%%%%%%%%%%%%%%%%%%%%%%%%%%%%%%%%%%%%%%%%%%%%%%%%%%%%%%

\bibitem[LZ16]{luzel} Z. Lu and S. Zelditch,
\newblock{\it{Szeg\'o kernels and Poinca\'re series}};
\newblock{Journal d'Analyse Math\'ematique, 130, 167–184, 2016}.
%%%%%%%%%%%%%%%%%%%%%%%%%%%%%%%%%%%%%%%%%%%%%%%%%%%%%%%%%%%%%%%%%%%%%%%%%%%%%%%%%%%%%%%%%%%

\bibitem[MM15]{ma} X. Ma and G. Marinescu,
\newblock{\it{Exponential estimate for the asymptotics of Bergman kernels}},
\newblock{Math. Ann. 362, 132--1347, 2015}.
%%%%%%%%%%%%%%%%%%%%%%%%%%%%%%%%%%%%%%%%%%%%%%%%%%%%%%%%%%%%%%%%%%%%%%%%%%%%%%%%%%%%%%%%%%%
%%%%%%%%%%%%%%%%%%%%%%%%%%%%%%%%%%%%%%%%%%%%%%%%%%%%%%%%%%%%%%%%%%%%%%%%%%%%%%%%%%%%%%%%%%%
 \end{thebibliography}
\end{document}